\documentclass[12pt]{amsart}
\usepackage{a4wide}
\usepackage{amssymb}
\usepackage{amsthm}
\usepackage{amsmath}
\usepackage{amscd}
\usepackage{verbatim}
\usepackage{upgreek}
\usepackage[all]{xy}
\usepackage{hyperref}
\usepackage{mathrsfs}
\usepackage{stmaryrd}

\numberwithin{equation}{section}

\theoremstyle{plain}
\newtheorem{theorem}{Theorem}[section]

\newtheorem{lemma}[theorem]{Lemma}

\theoremstyle{definition}
\newtheorem{definition}[theorem]{Definition}
\newtheorem{remark}[theorem]{Remark}

\theoremstyle{remark}

\newcommand{\R}{\mathbb{R}}

\newcommand{\C}{\mathbb{C}}

\newcommand{\Hom}{\operatorname{Hom}}
\newcommand{\Norm}{\operatorname{Norm}}

\newcommand{\Aut}{\operatorname{Aut}}

\newcommand{\Cent}{\operatorname{Cent}}

\newcommand{\der}{\operatorname{der}}

\newcommand{\tr}{\operatorname{tr}}

\newcommand{\Gal}{\operatorname{Gal}}

\newcommand{\reg}{\operatorname{reg}}

\newcommand{\disc}{\operatorname{disc}}

\newcommand{\temp}{\operatorname{temp}}

\newcommand{\elll}{\operatorname{ell}}

\newcommand{\dimm}{\operatorname{dim}}

\newcommand{\Out}{\operatorname{Out}}

\begin{document}

\title[The spectral side of stable local trace formula]{The spectral side of stable local trace formula for real groups}
\author[]{}
\dedicatory{}

\address{}
\email{}

\thanks{}

\date{\today}
\author[Chung Pang Mok, Zhifeng Peng*]{Chung Pang Mok, Zhifeng Peng*}

\address{Institute of Mathematics, Academia Sinica, Taipei}
\email{mokc@gate.sinica.edu.tw}

\address{Department of Mathematics, National University of Singapore}

\email{matpeng@nus.edu.sg}

\maketitle
\begin{abstract}
Let $G$ be a connected quasi-split reductive group over $\R$, and more generally, a quasi-split $K$-group over $\R$. Arthur had obtained the formal formula for the spectral side of the stable local trace formula, by using formal substitute of Langlands parameters. In this paper, we construct the spectral side of the stable trace formula and endoscopy trace formula directly for quasi-split $K$-groups over $\R$, by incorporating the works of Shelstad. In particular we give the explicit expression for the spectral side of the stable local trace formula, in terms of Langlands parameters. 
\end{abstract}

\maketitle


\section{Introduction}
\label{sect: introduction}
    
In this paper, which is a sequel to \cite{P}, we give the explicit formula for spectral side of stable local trace formula of a connected quasi-split reductive group over $\R$, and more generally, a quasi-split $K$-group over $\R$.

\bigskip

In general, the local trace formula is an identity, one side which is called the geometric side, is constructed by the semisimple orbital integrals; the other side, which is called the spectral side, is constructed in terms of tempered characters. Arthur \cite{A6} has obtained the stabilization of the geometric side, and consequently obtained the formal formula for the spectral side of the stable trace formula. However, the stable distributions and the coefficients that occurred in the formal formula for the spectral side, are not explicit. 

\bigskip

By combining with Shelstad's works \cite{S1,S2,S3}, we will directly stabilize the spectral side of the local trace formula, which in particular give the explicit formula for the spectral side of the stable local trace formula, in terms of Langlands parameters.\\

In more details, let $G$ be a quasi-split $K$-group over $\R$, a notion for which we refer to Section 1 of \cite{A4} (where it is called {\it multiple groups}) or section 2.2 of \cite{P}, and $f$ a test function on $G(\R)$ with central character $\zeta$. The endoscopic decomposition of the spectral side of the invariant local trace formula, as obtained in \cite{A6}, takes the following form:
\[
  I^G_{\disc}(f)=\sum_{G^{\prime}\in\mathcal{E}_{\elll}(G)}\iota(G,G^{\prime})\widehat{S}^{G^{\prime}}_{\disc}(f^{G^{\prime}})
  \]
where $\mathcal{E}_{\elll}(G)$ is the set of $\widehat{G}$-equivalence classes of elliptic endoscopic data, and $f^{G^{\prime}}$ is the Langlands-Shelstad transfer of $f$ to $G^{\prime}$ \cite{S1}. One has the formal formula:
   $$\widehat{S}^{G^{\prime}}_{\disc}(f^{G^{\prime}})=\int_{\Phi_{s-\disc}(G^{\prime},\zeta)}s^{G^{\prime}}(\phi^{\prime})f^{G^{\prime}}(\phi^{\prime})d\phi^{\prime}$$
with $\Phi_{s-\disc}(G^{\prime},\zeta)$ being defined only in terms of formal substitute of Langlands parameters $\phi^{\prime}$ of $G^{\prime}$, and the coefficients $s^{G^{\prime}}(\phi^{\prime})$ are not explicit. Our task is to show that $\Phi_{s-\disc}(G^{\prime},\zeta)$ can be taken as actual Langlands parameters, and also to give explicit formula for the coefficients in terms of Langlands parameters; {\it c.f.} Section 4 for the definition of these terms. \\

In the global context of automorphic representations, the method of \cite{A1, A9} is based on the comparison of the spectral and endoscopic objects for the global trace formula for $G$. This could be expressed schematically in terms of (conjectural) Langlands parameters:
     $$(M,\phi_{M})\longrightarrow(\phi,s)\longleftarrow (G^{\prime},\phi^{\prime})$$

\noindent In the archimedean local setting, the Langlands parametrization is available for general $G$. This allows us to adapt the comparison process to the setting of local trace formula:

     \[(\tau)\longrightarrow(\phi,s)\longleftarrow (G^{\prime},\phi^{\prime})\]
    in order to give the explicit construction of the spectral side of stable local trace formula.
     
\bigskip

We now give more details for the comparison process. Firstly recall the classification theory of tempered representations, due to Harish-Chandra. The tempered representations can be classified by triplets $\tau=(M,\pi,r)$, where $M$ is the Levi subgroup, $\pi\in \Pi_{2}(M)$ is square integrable modulo the split centre of $M$, and $r\in R_{\pi}$, the representation theoretic $R$-group of $\pi$, which is a finite abelian elementary $2$-group. For a test function $f= f_1 \times \bar{f}_2$, with $f_1,f_2 \in \mathcal{H}(G(\R),\zeta)$, the Hecke space with central character $\zeta$, the spectral side of the invariant local trace formula for $G$ takes the following form \cite{A2}:
    \[I^G_{\disc}(f)=\int_{T_{\disc}(G,\zeta)}i^{G}(\tau) f_{G}(\tau)|R_{\pi}|^{-1}d{\tau}.\]

\noindent where $f_{G}(\tau)= \Theta(\tau,f_1) \Theta(\tau^{\vee},\bar{f}_2)$, and the coefficient
    \[ i^{G}(\tau)=|W^{\circ}_{\pi}|^{-1}\sum_{w \in W_{\pi}(r)_{\reg}}\varepsilon_{\pi}(w)|\det(w-1)_{\mathfrak{a}^{G}_{M}}|^{-1},\]
encodes combinatorial data from Weyl groups that is relevant to the comparison of global trace formulas \cite{A1,A9}; {\it c.f.} Section 3 for the definition of these terms.\\

 The first step in the process of stabilization is to express the coefficient $i^G(\tau)$ in terms of  data defined in terms of Langlands parameters for $G$, and more precisely, defined in terms of the data $(\phi,s)$ above, an important point being that, by the works of Knapp-Zuckerman \cite{KZ} and Shelstad \cite{S1}, the representation theoretic $R$-group is canonically isomorphic to the endoscopic $R$-group defined in terms of Langlands parameters; {\it c.f.} Section 4. 
\[i^{G}(\tau)=|W^{\circ}_{\phi}|^{-1}\sum_{w\in W_{\phi}(x)_{\reg}}s^{\circ}_{\phi}(w)|\det(w-1)_{\mathfrak{a}^{G}_{M}}|^{-1}.\]
where $x$ is the image of $s$ in $\mathcal{S}_{\phi} = \pi_0(\overline{S}_{\phi})$. In turn, this can be expressed in terms of the constants $\sigma(\overline{S}_{\phi,s}^{\circ})$ as defined in \cite{A1}:
 \[ i^G(\tau) = i_{\phi}(x)= \sum_{s\in \mathcal{E}_{\phi,\elll}(x)}|\pi_{0}(\overline{S}_{\phi,s})|^{-1}\sigma(\overline{S}_{\phi,s}^{\circ}).\]

Following \cite{A1}, one defines the following subsets $\Phi_{s-\disc}(G,\zeta)\subset \Phi_{\disc}(G,\zeta)$ of the set $\Phi(G,\zeta)$ of Langlands parameters for $G$ (with central character $\zeta$), as:
    \[\Phi_{s-\disc}(G,\zeta)=\{\phi\in\Phi(G,\zeta):Z(\overline{S}_{\phi}^{\circ})<\infty\},\] 
 \[\Phi_{\disc}(G,\zeta)=\{\phi\in\Phi(G,\zeta):Z(\overline{S}_{\phi})<\infty\}.\] 
One has the fact that the constants $\sigma(\overline{S}_{\phi,s}^{\circ})$ vanish if $Z(\overline{S}_{\phi,s}^{\circ})$ is not finite; in particular that the constants $\sigma(\overline{S}_{\phi,s}^{\circ})$ vanish if $\phi \notin \Phi_{s-\disc}(G,\zeta)$.

\bigskip

Having defined the basic endoscopic and stable objects on the spectral side, we can define the spectral transfer factors, in the spirit of Section 5 of \cite{A3}, by combining the classification of tempered representations and Shelstad's definition of spectral transfer factors \cite{S2,S3}. For instance, in the case where $\phi \in \Phi(G,\zeta)$ is elliptic, we define, for $\tau =(M,\pi,r)$ and $s \in \overline{S}_{\phi}$ ({\it c.f.} Section 5):
        \[\Delta(\tau,\phi^{s})=\sum_{\chi\in\widehat{R}_{\phi}} \chi(r) \Delta(\Pi^{\chi},\phi^s).\]

We then have the spectral transfer:
     \[ \Theta(\tau,f)=\sum_{x \in\mathcal{S}_{\phi}}\Delta(\tau,\phi^{x})f^{\prime}(\phi,x).\]

We then obtain the first main theorem of the paper in Section 6:
 \begin{theorem}
If $f=f_{1}\times\bar{f}_{2}, f_{i}\in \mathcal{H}(G(\R),\zeta), i=1,2,$
 then
 \[I^G_{\disc}(f) =\int_{\Phi_{\disc}(G,\zeta)}\sum_{s\in \mathcal{E}_{\phi,\elll}}|\mathcal{S}_{\phi}|^{-1}|\pi_{0}(\overline{S}_{\phi,s})|^{-1}\sigma(\overline{S}_{\phi,s}^{\circ}) f^{\prime}_{1}(\phi,s)\overline{f^{\prime}_{2}(\phi,s)}d\phi .                 \]

\end{theorem}

To obtain the explicit formula for the spectral side of the endoscopy local trace formula and the stable local trace formula, we use the arguments of Chapter 4 of \cite{A9}; one analyzes the coefficients by using the bijective correspondence:
             \[ \widehat{G}  \bbslash X_{\disc}(G,\zeta) \longleftrightarrow  \widehat{G} \bbslash Y_{\disc}(G,\zeta) , \]
 between the set of $\widehat{G}$-conjugacy classes of $X_{\disc}(G,\zeta)=\{(\phi,s):\phi\in\Phi_{\disc}(G,\zeta), s\in S_{\phi,\elll}\}$, and the set of $\widehat{G}$-conjugacy classes of $Y_{\disc}(G,\zeta) = \{(G^{\prime},\phi^{\prime}):G^{\prime}\in\mathcal{E}_{\elll}(G),\phi^{\prime}\in\Phi_{s-\disc}(G^{\prime}, \zeta)  \}$ (see Section 7 for details).

\bigskip
 We then obtain the following in Section 7:
\begin{theorem}
 If $f=f_{1}\times \bar{f}_{2}$, $f_{i}\in \mathcal{H}(G(\R),\zeta),i=1,2$.
 Then we have
 \begin{equation}
 I^G_{\disc}(f)=\sum_{G^{\prime}\in\mathcal{E}_{\elll}(G)}\iota(G,G^{\prime}) \widehat{S}^{G^{\prime}}_{\disc}(f^{G^{\prime}})  \end{equation}
where
\[\iota(G,G^{\prime})=|\Out_{G}(G^{\prime})|^{-1}|\overline{Z}(\widehat{G^{\prime}})^{\Gamma_{\R}}|^{-1},\]
\[\widehat{S}^{G^{\prime}}_{\disc}(f^{G^{\prime}})=\int_{\Phi_{s-\disc}(G^{\prime},\zeta)}|\mathcal{S}_{\phi^{\prime}}|^{-1}\sigma(\overline{S}_{\phi^{\prime}}^{\circ})f_{1}^{G^{\prime}}(\phi^{\prime})\overline{f_{2}^{G^{\prime}}(\phi^{\prime})}d\phi^{\prime}.\]
 \end{theorem}

The explicit formula for the spectral side of the stable trace formula will then be obtained in Section 8.

\bigskip

 Here is the summary of the contents of the paper. After introducing preliminaries and notations in Section 2, we recall some formulation of the invariant local trace formula of Arthur in Section 3. Then in Section 4 we introduce the basic objects that occur in the spectral side of the endoscopic local trace formula and the stable trace formula. We will study the properties of the spectral transfer factors in Section 5. We then obtain the main theorem on stabilization of the local trace formula in Section 6. The explicit formula for the spectral side of the endoscopic and stable local trace formula are then obtained in Section 7 and 8.

\section*{Acknowledgement}

The authors would like to thank the referee for the careful reading of the manuscript. For this work, Zhifeng Peng had been supported by the National Natural Science Foundation of the People's Republic of China, Grant No.11601503.

\section{Preliminaries and notation}
\label{preliminaries and notation}
Throughout the paper $G$ is a $K$-group over $\R$, which will be assumed to be quasi-split from Section 4 onwards. The center of $G$ is noted as $Z(G)$, while $Z$ stands for a fixed central induced torus in $G$ over $\R$, and $\zeta$ is a character on $Z(\R)$. Let
  \[\mathfrak{a}_{G}=\Hom(X(G)_{\R},\R)\]
  be the real vector space dual of the module $X(G)_{\R}$ of $\R$ rational characters on $G$. There is a canonical homomorphism
   \[ H_{G}:G(\R)\rightarrow \mathfrak{a}_{G}\]
   defined by
    \[e^{\langle H_{G}(x),\chi\rangle}=|\chi(x)|, x\in G(\R),\chi\in X(G)_{\R},\] where $|\cdot|$ is the absolutely valuation on $\R$. Let $A_{G}$ be the split component of the center of $G$, then
   \[\mathfrak{a}_{G}=H_{G}(G(\R))=H_{G}(A_{G})\]
and $\mathfrak{a}_G$ is the Lie algebra of $A_G$. It is convenient to fix a Haar measure on $\mathfrak{a}_{G}$. This determines a dual Haar measure on the real vector space $i\mathfrak{a}^{\ast}_{G}$. It also determines a unique Haar measure on $A_{G}(\R)$. We denote $\mathfrak{a}^{\ast}_{G,Z}$ by the subspace of linear forms on $\mathfrak{a}_{G}$ that are trivial on the image of $\mathfrak{a}_{Z}$ in $\mathfrak{a}_{G}$. 

\bigskip

We recall some settings as in \cite{A2}. Denote by $\mathcal{L}^{G}(M)$ the finite set of Levi subgroups of $G$ which contain a given Levi subgroup $M$. We let $M_{0}$ be a fixed Levi component of some minimal parabolic subgroup of $G$. Put $\mathcal{L} := \mathcal{L}^G(M_0)$. For $M \in \mathcal{L}$, denote by $\mathcal{P}^G(M)$ the set of parabolic subgroups of $G$ having $M$ as Levi component. Define $\Pi_{2}(M(\R))$ to be the set of (equivalence classes of) representations that are square integrable modulo the split center of $M$, and $\Pi_{\temp}(M(\R))$ to be the set of (equivalence classes of) tempered representations. 

\bigskip
For $M\in \mathcal{L}^{G}(M_{0})$ and $\pi \in \Pi_{2}(M(\R))$, denote by $\Pi_{\pi}(G(\R))$ the set of irreducible constituents of the induced representation $I_{P}(\pi)$, which is a finite subset of $\Pi_{\temp}(G(\R))$ and independent of the parabolic subgroup $P$. The sets $\Pi_{\pi}(G(\R))$ exhaust $\Pi_{\temp}(G(\R))$. The classification of the representations in $\Pi_{\temp}(G(\R))$ is reduced to classifying the representations in the finite sets $\Pi_{\pi}(G(\R))$, and to determining the intersection of any two such sets. The second question is answered by Harish-Chandra's work.

\bigskip
    Write $W^{G}_{0}$ for the Weyl group of the pair $(G,A_{M_0})$; for $w \in W_0^G$, we generally write $\tilde{w}$ for any representative of $w$ in $K$; here $K$ is a (fixed) maximal compact subgroup $G(\R)$ that is in good position relative to $M_0(\R)$. If $M\in \mathcal{L}^{G}(M_{0})$ and $\pi \in \Pi_{2}(M(\R))$, $wM=\tilde{w}M\tilde{w}^{-1}$ is another Levi subgroup, and
    \[(w \pi)(m^{\prime})=\pi(\tilde{w}^{-1}m^{\prime}\tilde{w}), \quad m^{\prime}\in(wM)(\R)\]
   is a representation in $\Pi_{2}((wM)(\R))$. We obtain an action
   \[(M,\pi)\rightarrow(wM, w\pi), w\in W^{G}_{0},\]
   of $W^{G}_{0}$ on the set of pairs $(M,\pi), M \in\mathcal{L}^{G}(M_{0}),\pi \in\Pi_{2}(M(\R)).$

\bigskip
As stated in Proposition 1.1 of \cite{A2}, one has the following: If $(M,\pi)$ and $(M^{\prime},\pi^{\prime})$ are any two pairs, and $(M^{\prime},\pi^{\prime})$ equals $(wM,w\pi)$ for an element $w\in W^{G}_{0},$ the subsets $\Pi_{\pi}(G(\R))$ and $\Pi_{\pi^{\prime}}(G(\R))$ of $\Pi_{\temp}(G(\R))$ are identical. Conversely, if the sets $\Pi_{\pi}(G(\R))$ and $\Pi_{\pi^{\prime}}(G(\R))$ have a representation in common, there is an element $w\in W^{G}_{0}$ such that $(M^{\prime},\pi^{\prime})=(wM,w\pi)$.

\bigskip

In this paper we shall fix the central data $(Z,\zeta)$. Thus $\mathcal{H}(G(\R),\zeta)$ is the Hecke space of smooth functions with compact support $f$ on $G(\R)$ that are left and right finite under the maximal compact subgroup $K$, and such that $f(zx) = \zeta(z)^{-1} f(x)$ for $z \in Z(\R)$ and $x \in G(\R)$. Similarly we define $\Pi_{\temp}(G(\R),\zeta)$ to be the subset of $\Pi_{\temp}(G(\R))$ consisting of those representations whose character on $Z(\R)$ is equal to $\zeta$. Similarly one defines the subsets $\Pi_2(G(\R),\zeta) := \Pi_2(G(\R)) \cap \Pi_{\temp}(G(\R),\zeta)$ and $\Pi_{\pi}(G(\R),\zeta) := \Pi_{\pi}(G(\R)) \cap \Pi_{\temp}(G(\R),\zeta)$ etc.

\bigskip

We can understand the finite set $\Pi_{\pi}(G(\R),\zeta)$ by the representation theoretic $R$-group \cite{KZ}. More generally, we can parameterize the characters of the tempered representations $\Pi_{\temp}(G(\R),\zeta)$ by the virtual characters of the sets $T(G,\zeta)$, where $T(G,\zeta)$ is the set of $G$-equivalence classes of the sets $\widetilde{T}(G,\zeta)$, with $\widetilde{T}(G,\zeta)$ being defined as $\widetilde{T}(G,\zeta)=\{\tau=(M,\pi,r): M\in \mathcal{L}^G(M_{0}),\pi\in \Pi_{2}(M(\R),\zeta),r\in R_{\pi} \}$. Here $\Pi_{2}(M(\R),\zeta)$ is as before the subset of $\Pi_2(M(\R))$ consisting of those $\pi$ whose character on $Z(\R)$ is equal to $\zeta$. Finally the representation theoretic $R$-group $R_{\pi}$ of $\pi$ is defined as the quotient of $W_{\pi}$ by
   $W^{\circ}_{\pi}$, where \[W_{\pi}=\{w\in W(\mathfrak{a}_{M}):w\pi\cong\pi\}\] is the stabilizer of $\pi$ in the Weyl group of $\mathfrak{a}_{M}$, and $W^{\circ}_{\pi}$ is the subgroup of elements $w$ in $W_{\pi}$ such that the operator $R(w,\pi)$ is a scalar ({\it c.f.} below). It is known that $W^{\circ}_{\pi}$ is a normal subgroup of $W_{\pi}$. The group $W_{\pi}^{\circ}$ is the Weyl group of a root system, composed of scalar multiples of those reduced roots $\alpha$ of $(G,A_{M})$ for which the reflection $w_{\alpha}$ belongs to $W^{\circ}_{\pi}$. These roots divide the vector space $\mathfrak{a}_{M}$ into chambers. By fixing such a chamber $\mathfrak{a}_{\pi}$, we can identify $R_{\pi}$ with the subgroup of elements in $W_{\pi}$ that preserve $\mathfrak{a}_{\pi}$.

\bigskip
     The operator
     \[R(w,\pi)=A(\pi_{w})R_{\tilde{w}^{-1}P\tilde{w}|P}(\pi), \quad w\in W_{\pi}, \pi\in \Pi_{2}(M), P\in \mathcal{P}^G(M) \]
 is an intertwining operator from $I_{P}(\pi)$ to itself. Here to define $A(\pi_w)$, first note that $\pi$ can be extended to a representation of the group $M^{\prime}(\R)$ generated by $M(\R)$ and $\tilde{w}$. We denote $\pi_{w}$ by such an extension, then the intertwining operator
        \[  A(\pi_{w}): I_{\tilde{w}^{-1}P\tilde{w}}(\pi)\longrightarrow I_{P}(\pi),  \pi\in W_{\pi}\]
between $I_{\tilde{w}^{-1}P\tilde{w}}(\pi)$ and $I_{P}(\pi)$ is defined by setting
   \[ (A(\pi_{w})\phi^{\prime})(x)=\pi_{w}(\tilde{w})\phi^{\prime}(\tilde{w}^{-1}x), \phi^{\prime}\in I_{\tilde{w}^{-1}P\tilde{w}}(\pi). \]
 Finally the intertwining operator $R_{\tilde{w}^{-1}P\tilde{w}|P}(\pi)$, and more generally 
    \[R_{Q|P}(\pi)=r_{Q|P}(\pi)^{-1}J_{Q|P}(\pi): I_{P}(\pi)\longrightarrow I_{Q}(\pi),\mbox{ for }P,Q\in \mathcal{P}^G(M),\]
is the normalized intertwining operator between the induced representations $I_{P}(\pi)$ and $I_{Q}(\pi)$, with $r_{Q|P}(\pi)$ being the normalizing factors; see the discussion on p. 85 of \cite{A2}. In addition, since we are in the archimedean case, the operators $R(w,\pi)$ can be normalized so that $R(w,\pi)$ is the identity for $w \in W_{\pi}^{\circ}$ (hence $R(r,\pi)$ is well-defined for $r \in R_{\pi}$), and such that the map $r \mapsto R(r,\pi)$ is a homomorphism on $R_{\pi}$, {\it c.f.} p. 86 of \cite{A2}. We will always work with such a normalization in what follows.

  \section{Local trace formula}

In this section we recall the formalism of the local trace formula of Arthur \cite{A2}, specifically in the archimedean case. Consider a test function $f=f_{1}\times \bar{f}_{2}$, with $f_{i}\in \mathcal{H}(G(\R),\zeta)$ being in the Hecke space of test functions. Following Arthur, it is customary to think of the components $f_1,f_2$ of $f$ as being indexed by the two elements set $V=\{\infty_1,\infty_2\}$ (being regarded as two archimedean places). The local trace formula is given by the identity $I^G_{\disc}(f) =I^G(f)$, {\it c.f.} Theorem 4.2 of \cite{A2} and Proposition 6.1 of \cite{A6}. The geometric side of the local trace formula is given by:
   \[  I^G(f)=\sum_{M\in\mathcal{L}}|W^{M}_{0}||W^{G}_{0}|^{-1}(-1)^{\dimm(A_{M}/ A_{G})}\int_{\Gamma_{G-\reg,\elll}(M,V,\zeta)}I^G_{M}(\gamma,f) d\gamma\]
   defined in terms of the invariant distributions $I^G_{M}(\gamma,f)$ ({\it c.f. loc. cit.}) For the definition of $\Gamma_{G-\reg,\elll}(M,V,\zeta)$, firstly, for a fixed basis $\Gamma(M,\zeta)$ of the space of invariant distributions $\mathcal{D}(M,\zeta)$ on $M(\R)$ introduced in Section 1 of \cite{A5} (which in particular are $\zeta$-equivariant under translation by $Z(\R)$), one has the subset $ \Gamma_{G-\reg,\elll}(M,\zeta)$ of $\Gamma(M,\zeta)$ consisting of elements that are strongly $G$-regular, elliptic support in $M(\R)$. Put $\Gamma(M_V,\zeta_V)= \Gamma(M,\zeta) \times \Gamma(M,\zeta)$ (corresponding to the two places $\infty_1,\infty_2$ in $V$). Then $\Gamma_{G-\reg,\elll}(M,V,\zeta)$ is identified with the diagonal image of $\Gamma_{G-\reg,\elll}(M,\zeta)$ in $\Gamma(M_V,\zeta_V)$:
\[
\Gamma_{G-\reg,\elll}(M,V,\zeta)  = \{(\gamma,\gamma):\gamma \in \Gamma_{G-\reg,\elll}(M,\zeta)\}.
 \]   

\bigskip
The other side is the spectral side given by:
     \[I^G_{\disc}(f)=\int_{T_{\disc}(G,\zeta)}i^{G}(\tau) f_{G}(\tau) |R_{\pi}|^{-1} d\tau. \]
Here: \[T_{\disc}(G,\zeta)=\{\tau=(M,\pi,r)\in T(G,\zeta): W_{\pi}(r)_{\reg}\neq\emptyset, \pi\in\Pi_{2}(M(\R),\zeta)\}, \]

  $W_{\pi}(r)_{\reg}=W_{\pi}(r)\cap W_{\pi,\reg}, W_{\pi}(r)=W_{\pi}^{\circ} \cdot r$ , $W_{\pi,\reg}=\{w\in W_{\pi}:\mathfrak{a}^{w}_{M}=\mathfrak{a}_{G}\},$ and
  \[i^{G}(\tau)=|W^{\circ}_{\pi}|^{-1}\sum_{w\in W_{\pi}(r)_{\reg}}\varepsilon_{\pi}(w)|\det(w-1)_{\mathfrak{a}^{G}_{M}}|^{-1}\]
  which encode combinatorial data from Weyl groups. The sign $\varepsilon_{\pi}(w)$ stands for the sign of projection of $w$ onto the Weyl group $W^{\circ}_{\pi}$ taken relative to the decomposition $W_{\pi}=W_{\pi}^{\circ}\rtimes R_{\pi}$, and $\mathfrak{a}_M^G$ is the quotient of $\mathfrak{a}_M$ by $\mathfrak{a}_G$. 

\bigskip
The other terms are defined as:

   \[f_{G}(\tau)=\Theta(\tau,f_1) \Theta(\tau^{\vee},\bar{f}_2)   =  \Theta(\tau,f_1) \overline{\Theta(\tau,f_2) } \]

\[ \Theta(\tau,f_i)=\tr(R(r,\pi)I_{P}(\pi,f_{i})), \,\ i=1,2.\]

\bigskip

Finally the measure $d\tau$ on $T_{\disc}(G,\zeta)$ is defined by the formula ({\it c.f.} equation (3.5) of \cite{A2}):

 \begin{eqnarray*} 
 \int_{T_{\disc}(G,\zeta)}f(\tau)d\tau &=&\sum_{\tau\in T_{\disc}(G,\zeta)/ i\mathfrak{a}^{\ast}_{G,Z}} \int_{i\mathfrak{a}^{\ast}_{G,Z}}f(\tau_{\lambda})d\lambda.\end{eqnarray*}
for $f \in C_c(T_{\disc}(G,\zeta))$.

\bigskip

Here when compared to equation (3.5) of \cite{A2}, we first note that since we are in the archimedean case, the groups $R_{\pi}$ are abelian, and hence the groups $R_{\pi,r}$ in {\it loc. cit.}, namely the centralizer of $r$ in $R_{\pi}$, reduces to $R_{\pi}$.

\bigskip

Secondly, for our purpose, it would be more convenient to {\it not} to absorb the factor $|R_{\pi,r}|^{-1}=|R_{\pi}|^{-1}$ into the definition of the measure $d \tau$ on $T_{\disc}(G,\zeta)$, as was done in \cite{A2}.

  \section{Endoscopic and stable objects on the spectral side}

  From now on $G$ will always be assumed to be a quasi-split $K$-group. We first recall the Langlands parameters. These are admissible continuous homomorphisms:
    \[\phi: W_{\R} \longrightarrow {}^{L}G \]
where as usual ${}^{L}G  = \widehat{G}  \rtimes W_{\R}$ is the $L$-group of $G$, defined with respect to a splitting of $G$ (that we fix for the rest of the paper); $W_{\R}$ is the Weil group of $\R$: it is a non-split extension $1 \rightarrow \C^{\times} \rightarrow W_{\R} \rightarrow \Gamma_{\R} \rightarrow 1$, with $\Gamma_{\R}=\Gal(\C/\R)$. The parameter $\phi$ is called bounded, if the image of $W_{\R}$ in $\widehat{G}$ is bounded. We denote by $\Phi(G)$ for the set of $\widehat{G}$-equivalence classes of bounded parameters (with respect to the conjugation action by $\widehat{G}$). For $\phi \in \Phi(G)$, we denote by $\Pi_{\phi}$ the $L$-packet of tempered representations of $G(\R)$ associated to $\phi$. The stable character $f \mapsto f(\phi):=\sum_{\Pi \in \Pi_{\phi}} \tr \Pi(f) $ is then a stable distribution on $G(\R)$ \cite{S1}.
\bigskip

We denote by $\Phi_{2}(G)$, the set of (equivalence classes of) square-integrable parameters, for the subset of $\phi\in\Phi(G)$ that does not factor through ${}^{L}M$, for any proper Levi subgroup $M$ of $G$. For $\phi \in \Phi_2(G)$, the $L$-packet $\Pi_{\phi}$ consists of square-integrable representations of $G(\R)$.

\bigskip
For any $\phi\in \Phi(G)$, we set
    \[S_{\phi}=\Cent(Im\phi,\widehat{G}),\] \[\overline{S}_{\phi}=S_{\phi}/ Z(\widehat{G})^{\Gamma_{\R}},\] 
and \[\mathcal{S}_{\phi}=\pi_{0}(\overline{S}_{\phi}).\]
Since we are in the archimedean case, the component group $\mathcal{S}_{\phi}$ is a finite abelian elementary $2$-group \cite{S1}. In addition, since we are working in the context of a quasi-split $K$-group, one has that $\Pi_{\phi}$ is in bijection with the set of characters of $\mathcal{S}_{\phi}$; hence the cardinality of $\Pi_{\phi}$ is equal to the order of $\mathcal{S}_{\phi}$ \cite{S1,S3}.

\bigskip

\noindent One has
\[
\Phi_2(G) = \{\phi \in \Phi(G), |\overline{S}_{\phi}| < \infty \}.
\]

\bigskip

\noindent We also define the following subsets of parameters $ \Phi_{\disc}(G) , \Phi_{s-\disc}(G), \Phi_{\elll}(G),$ of $\Phi(G)$:
\[\Phi_{\disc}(G)=\{\phi \in \Phi(G),|Z(\overline{S}_{\phi})|<\infty\}, \]
with $Z(\overline{S}_{\phi}) := \Cent(\overline{S}_{\phi},\overline{S}_{\phi}^{\circ})$, and:

\[\Phi_{s-\disc}(G)=\{\phi \in \Phi(G),|Z(\overline{S}^{\circ}_{\phi})|<\infty\},\]
with $Z(\overline{S}_{\phi}^{\circ})$ being the usual center of $\overline{S}_{\phi}^{\circ}$.

\bigskip
\noindent The set of elliptic parameters $\Phi_{\elll}(G)$ is defined as the subset of $\phi \in \Phi(G)$, such that $|\overline{S}_{\phi,s}|<\infty$ for some semi-simple element $s \in \overline{S}_{\phi}$; here $\overline{S}_{\phi,s}$ is being defined as:
\[
\overline{S}_{\phi,s} := \Cent(s,\overline{S}_{\phi}^{\circ}  ).
\]

One has:
\[
\Phi_2(G) \subset \Phi_{s-\disc}(G) \subset \Phi_{\disc}(G),
\]

\[
\Phi_2(G) \subset \Phi_{\elll}(G) \subset \Phi_{\disc}(G).
\]

Also, for a central data $(Z,\zeta)$ of $G$ as before, we denote by $\Phi(G,\zeta)$ the set of parameters $\phi \in \Phi(G)$ that have character $\zeta$ with respect to $Z$, in the sense that the composition:
\[
W_{\R} \stackrel{\phi}{\rightarrow} {}^{L} G \rightarrow {}^{L} Z 
\]
corresponds to the character $\zeta$ of $Z$. Similar definition for the sets $\Phi_2(G,\zeta),\Phi_{\elll}(G,\zeta)$ etc.

\bigskip
We now define the endoscopic $R$-group, \[R_{\phi} :=W_{\phi}/ W^{\circ}_{\phi},\]
 where the Weyl groups $W_{\phi},W_{\phi}^{\circ}$ are defined as:
\[W_{\phi}=\Norm(\overline{\mathcal{T}}_{\phi},\overline{S}_{\phi})/\overline{\mathcal{T}}_{\phi}.
\] 
Here $\overline{\mathcal{T}}_{\phi}$ is defined as $A_{\widehat{M}}/(A_{\widehat{M}}\cap Z(\widehat{G})^{\Gamma_{\R}})$, with $A_{\widehat{M}}=(Z(\widehat{M})^{\Gamma_{\R}})^{\circ}$, and $M$ being the Levi subgroup of $G$ (which is unique up to conjugation) such that $\phi$ factors through ${}^{L}M$ as a square integrable parameter $\phi_M \in \Phi_2(M,\zeta)$ of $M$. Similarly
    \[W_{\phi}^{\circ}=\Norm(\overline{\mathcal{T}}_{\phi},\overline{S}^{\circ}_{\phi})/\overline{\mathcal{T}}_{\phi}. \]
  
\bigskip

One also has, by the results in section 5 of \cite{S1}, the split short exact sequence:
\begin{eqnarray}
0 \rightarrow \mathcal{S}_{\phi_M} \rightarrow \mathcal{S}_{\phi} \rightarrow R_{\phi} \rightarrow 0.
\end{eqnarray}

\bigskip

For $\phi \in \Phi(G,\zeta)$, we denote by 
\[T_{\phi}=\{\tau=(M,\pi,r) \in T(G,\zeta), \mbox{ such that } \pi\in\Pi_{\phi_{M}} \} .\]
 \begin{lemma}
If $\phi\in\Phi(G,\zeta)$, then we have canonical identification $R_{\pi}=R_{\phi}, W_{\pi}=W_{\phi}, W^{\circ}_{\pi}=W^{\circ}_{\phi}$, for $\tau = (M,\pi,r) \in T_{\phi}$. 

\bigskip

Thus we have a natural surjective map of sets $T_{\phi} \rightarrow R_{\phi}$ by sending $\tau = (M,\pi,r)$ to $r \in R_{\pi} = R_{\phi}$.

\bigskip

We can also define a non-canonical bijection $\iota:T_{\phi} \rightarrow \mathcal{S}_{\phi}$ (which we fix once and for all) that respects the natural projection map to $R_{\phi}$.
\end{lemma}
\begin{proof}
This follow from the results of Knapp-Zuckerman \cite{KZ} and Shelstad \cite{S1}. If $\phi \in \Phi_{\elll}(G,\zeta)$ is elliptic, see the discussion before Proposition 5.2 of \cite{P}. In general, if $\phi\in\Phi(G,\zeta)$, then there is a Levi subgroup $\widetilde{M}$ of $G$ (unique up to conjugacy), that is maximal with respect to the property that the parameter $\phi$ factors through ${}^{L}\widetilde{M}$ as an elliptic parameter $\phi^{\widetilde{M}}$ of $\widetilde{M}$ (see section 5 of \cite{S1}). Under the Levi embedding ${}^{L} \widetilde{M} \hookrightarrow {}^{L}G$ on the dual side, have a canonical isomorphism \[\mathcal{S}_{\phi^{\widetilde{M}}}\longrightarrow\mathcal{S}_{\phi}.\]
   However $\phi^{\widetilde{M}}$ is elliptic, and so the lemma is true for $\phi^{\widetilde{M}}$; thus we have canonical isomorphisms (noting that the $R$-group of $\pi$ with respect to $G$ is canonically identified with the $R$-group of $\pi$ with respect to $\widetilde{M}$, {\it c.f. loc. cit.}): 
   \[ R_{\pi}=R_{\phi^{\widetilde{M}}}=R_{\phi}.\]

\bigskip

To show that one has canonical identifications between $W_{\pi},W_{\phi}$ and $W_{\pi}^{\circ},W_{\phi}^{\circ}$, first note that $W_{\phi}$ is the stabilizer of $\phi_{M}$ in $W(\mathfrak{a}_M)$. It is then a consequence of the disjointness of tempered $L$-packets that $W_{\phi}$ contains $W_{\pi}$. On the other hand, we know that the intertwining operators on $I_{P}(\pi)$ coming from the subgroup $W^{\circ}_{\phi}$ of $W_{\phi}$ are scalars ({\it c.f.} section 5 of \cite{S1}). It follows that $W^{\circ}_{\phi}$ is contained in the subgroup $W^{\circ}_{\pi}$ of $W_{\pi}$. Now $R_{\pi}=W_{\pi}/W^{\circ}_{\pi}, R_{\pi}=W_{\phi}/W^{\circ}_{\phi}$, and we already know that $|R_{\pi}|=|R_{\phi}|$, so it follows that we must have $W_{\pi}=W_{\phi}$ and $W^{\circ}_{\pi}=W^{\circ}_{\phi}.$

\bigskip

Finally, one can thus constructs a bijection $\iota:T_{\phi} \rightarrow \mathcal{S}_{\phi}$ that respects the projection to $R_{\phi}$, from a bijection between $T_{\phi^{\widetilde{M}}}$ and $\mathcal{S}_{\phi^{\widetilde{M}}}$ that respects the projection to $R_{\phi^{\widetilde{M}}}$.

\end{proof}

\bigskip
Next we recall some generalities from \cite{A1}. We shall consider $S$ to be any connected component of a complex reductive group. Given such $S$, we denote by $S^+$ the complex reductive group generated by $S$, and by $S^{\circ}$ the identity connected component of $S^+$. Put \[Z(S)=\Cent(S,S^{\circ})\] for the centralizer of $S$ in $S^{\circ}$. Then for a choice of maximal torus $\mathcal{T}$ of $S^{\circ}$, denote by $W(S^{\circ}) = \Norm(\mathcal{T},S^{\circ})/\mathcal{T}$ the usual Weyl group of $S^{\circ}$. We can also form the Weyl set \[W(S)=\Norm(\mathcal{T},S)/ \mathcal{T}.\] Denote by $W(S)_{\reg}$ the subset of elements $w\in W(S)$ that are regular, which means that the fixed point set of $w$ in $\mathcal{T}$ is finite. We can also regard $(w-1)$ as a linear transformation on the real vector space
  \[\mathfrak{a}_{\mathcal{T}}=\Hom(X(\mathcal{T}),\R).\]
Then the condition for $w \in W(S)_{\reg}$ is equivalent to that $\det(w-1)_{\mathfrak{a}_{\mathcal{T}}} \neq 0$.

\bigskip

We denote by:
      \[s^{\circ}(w)=\pm 1\]
    the sign of a element $w\in W$, to be $-1$ raised to the power of the number of the positive roots of $(S^{\circ},\mathcal{T})$ (with respect to some order) being mapped by $w$ to the negative roots. We then define the rational number
      \[i(S)=|W(S^{\circ})|^{-1}\sum_{w\in W(S)_{\reg}}s^{\circ}(w)|\det(w-1)_{\mathfrak{a}_{\mathcal{T}}}|^{-1}\]
associated to $S$.

\bigskip

Next we write $S_{ss}$ for the set of semisimple elements in $S$. For any $s\in S_{ss}$, we set \[S_{s}=\Cent(s,S^{\circ})\]
      the centralizer of $s$ in $S^{\circ}$. Then $S_{s}$ is also a complex reductive group, whose identity component is noted as: \[S^{\circ}_{s}=(S_{s})^{\circ}=\Cent(s,S^{\circ})^{\circ}.\]
      If $\Gamma$ is any subset of $S$ which is invariant under conjugation by $S^{\circ}$, then we shall denote by $\mathcal{E}(\Gamma)$ for the set of equivalence classes in $\Gamma_{ss}=\Gamma\cap S_{ss}$, with the equivalence relation defined by setting $s^{\prime}\sim s$ if \[s^{\prime}=s^{\circ}z s (s^{\circ})^{-1}, s^{\circ}\in S^{\circ},z\in Z(S^{\circ}_{s})^{\circ}.\]
      The main interest is the subset
      \[ S_{\elll}=\{s\in S_{ss}:|Z(S^{\circ}_{s})|<\infty\}\]
      of elliptic elements of $S$. The equivalence relation on $S_{\elll}$ is then simply $S^{\circ}$-conjugation. Put:
 \[\mathcal{E}_{\elll}(S):=\mathcal{E}(S_{\elll}).\]

\bigskip

      We have the following theorem which is a restatement of theorem $8.1$ of \cite{A1}.

      \begin{theorem} \label{coefficient endos theorem}
        There are unique constants $\sigma(S_{1})$, defined whenever $S_{1}$ is a connected complex reductive group, such that for any $S$ as above, the number
        \begin{equation}\label{coefid}
        e(S)=\sum_{s\in\mathcal{E}_{\elll}(S)}|\pi_{0}(S_{s})|^{-1}\sigma(S^{\circ}_{s})
        \end{equation}
        equals $i(S)$, and such that
        \begin{equation}\label{centv}
        \sigma(S_{1})=\sigma(S_{1} / Z_{1})|Z_{1}|^{-1}
      \end{equation}
        for any central subgroup $Z_{1}$ of $S_{1}$ (in particular $\sigma(S_1)=0$ if $Z_1$ is infinite).

      \end{theorem}

\bigskip

Now back to the situation of Lemma 4.1. For $\tau \in T_{\phi}$, put $x = \iota(\tau)$. Then we can write the coefficient $i^G(\tau)$ as:

\begin{eqnarray}
i^{G}(\tau) &=& i_{\phi}(x)=|W^{\circ}_{\phi}|^{-1}\sum_{w\in W_{\phi}(x)_{\reg}}s^{\circ}_{\phi}(w)|\det(w-1)_{\mathfrak{a}^{G}_{M}}|^{-1} \\
&=  & e_{\phi}(x) = \sum_{s\in \mathcal{E}_{\phi,\elll}(x)}|\pi_{0}(\overline{S}_{\phi,s})|^{-1}\sigma(\overline{S}_{\phi,s}^{\circ}) \nonumber
\end{eqnarray}

\noindent i.e. $i_{\phi}(x)$ is equal to the number $i(S_x)$ above, with $S_x$ being equal to the component of $\overline{S}_{\phi}$ that corresponds to $x \in \mathcal{S}_{\phi} = \pi_0(\overline{S}_{\phi})$, and $W_{\phi}(x)_{\reg}$ is the set $W(S_x)_{\reg}$ as defined above; similarly $e_{\phi}(x)$ is the number $e(S_x)$, with $\mathcal{E}_{\phi,\elll}(x)$ (resp. $\mathcal{E}_{\phi,\elll}$) being the set $\mathcal{E}_{\elll}(S_x)$ (resp. $\mathcal{E}_{\elll}(\overline{S}_{\phi})$) defined as above, etc.

\bigskip

Finally we note the following. Given $\phi \in \Phi(G,\zeta)$ as before, that factors through ${}^L M$ as a square-integrable parameter $\phi_M$ of Levi subgroup $M$ of $G$. Recall that one has a split short exact sequence:
\begin{eqnarray*}
0 \rightarrow \mathcal{S}_{\phi_M} \rightarrow \mathcal{S}_{\phi} \rightarrow R_{\phi} \rightarrow 0.
\end{eqnarray*}
Then for $x,y \in \mathcal{S}_{\phi}$, one has $i_{\phi}(x) = i_{\phi}(y)$ if $x = y \mod{\mathcal{S}_{\phi_M}}$. This follows easily from the definition of the number $i_{\phi}(x)$.

\bigskip

      The constants $\sigma(\overline{S}_{\phi,s}^{\circ})$, for $s \in \mathcal{E}_{\phi,\elll} = \mathcal{E}_{\elll}(\overline{S}_{\phi})$, appear as the coefficients of the endoscopic and stable local trace formula, as we are going to see in the following sections.

      \section{Spectral transfer}

        Suppose $({G^{\prime},s^{\prime},\mathcal{G}^{\prime},\xi^{\prime}})$ is an endoscopic datum for $G$. In the theory of endoscopy one chooses a $Z$-pair $(G^{\prime}_{1},\xi_{1}^{\prime})$, where $G^{\prime}_1$ is a $Z$-extension of $G^{\prime}$ and $\xi_1^{\prime}$ is an embedding of extensions $\mathcal{G}^{\prime} \hookrightarrow{}^{L}G_{1}^{\prime}$ that extends the embedding $\widehat{G}^{\prime} \hookrightarrow \widehat{G}^{\prime}_1$ dual to the surjection $G_1^{\prime} \rightarrow G^{\prime}$.

\bigskip

Given the endoscopic datum $({G^{\prime},s^{\prime},\mathcal{G}^{\prime},\xi^{\prime}})$ for $G$ as above, suppose that $\mathcal{G}^{\prime} = {}^{L}G^{\prime}$ (and thus $\xi^{\prime}: {}^{L}G^{\prime} \hookrightarrow {}^{L}G$ is an embedding of $L$-groups, and it is not needed to choose a $Z$-extension for $G^{\prime}$), we define a mapping \[\Phi(G^{\prime},\zeta)\longrightarrow \Phi(G,\zeta)\] by $\phi^{\prime}\mapsto \phi=\xi^{\prime} \circ\phi^{\prime}$.

\bigskip

For the general case, we refer to Section 2 of \cite{A7} and also Section 2 of \cite{S2}. This construction gives a correspondence $(G^{\prime},\phi^{\prime}) \longleftrightarrow (\phi,s)$ between pairs $(\phi,s)$, where $\phi \in \Phi(G,\zeta), s \in \overline{S}_{\phi}$ semi-simple, and pairs $(G^{\prime},\phi^{\prime})$, where $G^{\prime}=({G^{\prime},s^{\prime},\mathcal{G}^{\prime},\xi^{\prime}})$ is an endoscopic datum of $G$, and $\phi^{\prime} \in \Phi(G^{\prime},\zeta)$, {\it c.f. loc. cit.} For simplicity we always assume that $\mathcal{G}^{\prime} = {}^{L}G^{\prime}$ in what follows.

        \begin{definition}
   For $\phi^{\prime} \in \Phi(G^{\prime},\zeta)$ and $\Pi \in \Pi_{\temp}(G,\zeta)$, we say that $(\phi^{\prime},\Pi)$ is a related pair if $\phi(\Pi)$ is the image of $\phi^{\prime}$ under the map $\Phi(G^{\prime},\zeta)\longrightarrow \Phi(G,\zeta)$ associated to $\xi^{\prime}$; here $\phi(\Pi) \in \Phi(G,\zeta)$ is the Langlands parameter of $\Pi$.
        \end{definition}
        Given any $\phi^{\prime} \in \Phi(G^{\prime},\zeta)$, there is always a $\Pi \in \Pi_{\temp}(G,\zeta)$ that is related to $\phi^{\prime}$.

        \begin{definition}
        A related pair $(\phi^{\prime},\Pi)$ is $G$-regular if the parameter $\phi=\phi(\Pi)$ is $G$-regular, in the sense that we have that
          \[\Cent(\phi(\C^{\times}),\widehat{G})\] is abelian.
        \end{definition}

        Shelstad had built the spectral transfer factors $\Delta(\phi^{\prime},\Pi)$ \cite{S2} directly when the related pair $(\phi^{\prime},\Pi)$ is $G$-regular, and obtained spectral transfer identities (if the pair $(\phi^{\prime},\Pi)$ is not related, then one simply defines $\Delta(\phi^{\prime},\Pi)=0$). The general case is handled by using character identities of Hecht and Schmid, and coherent continuation of the identities from the $G$-regular case \cite{S2}.

\bigskip

 We have the following spectral transfer relations. For each $f\in \mathcal{H}(G(\R),\zeta)$, and for any endoscopic datum $G^{\prime}$ of $G$, there exists $f^{\prime}\in \mathcal{H}(G^{\prime}(\R),\zeta)$ such that the stable orbital integral of $f^{\prime}$ is equal to $f^{G^{\prime}}$, the Langlands-Shelstad transfer of $f$ to $G^{\prime}$ \cite{S1} (with respect to Whittaker normalization). Note that in particular when we take $G^{\prime}=G$, then $f^G$ is the stable orbital integral of $f$. In addition the following holds: for any endoscopic datum $G^{\prime}$ of $G$ and any tempered Langlands parameter $\phi^{\prime}$ of $G^{\prime}(\R)$, the stable character $f^{\prime}(\phi^{\prime})$ of $\phi^{\prime}$ (evaluated at $f^{\prime}$) satisfies:
          \[f^{\prime}(\phi^{\prime})=\sum_{\Pi\in\Pi_{\temp}(G,\zeta)}\Delta(\phi^{\prime},\Pi)\tr\Pi(f).\]
          
          \bigskip
Remark that, as the character $f^{\prime} \mapsto f^{\prime}(\phi^{\prime})$ is stable, it depends only on the stable orbital integral of $f^{\prime}$. Hence in the above, namely when the stable orbital integral of $f^{\prime}$ is equal to the Langlands-Sehlstad transfer $f^{G^{\prime}}$, the value $f^{\prime}(\phi^{\prime})$ depends only on $f^{G^{\prime}}$, and we will write $f^{\prime}(\phi^{\prime})$ as $f^{G^{\prime}}(\phi^{\prime})$.          
          
\bigskip

Shelstad ({\it c.f.} \cite{S3}, section 11) had also checked that the spectral transfer factors can be normalized (namely Whittaker normalization) to satisfy the Arthur's conjecture, and we will always work with such normalization in what follows; more precisely given endoscopic data $(G^{\prime},s^{\prime},\mathcal{G}^{\prime},\xi^{\prime})$ as above, one has, whenever $(\phi^{\prime},\Pi)$ is a related pair, that the number $\Delta(\phi^{\prime},\Pi)$ depends only on $\Pi$ and on the image $x_s$ of $s$ in $\mathcal{S}_{\phi}$ (here $\phi$ is the Langlands parameter of $\Pi$). 

\bigskip

Thus given $\phi \in \Phi(G,\zeta)$, and $s \in \overline{S}_{\phi}$ semi-simple, for $\Pi \in \Pi_{\phi}$, we will denote 
\[
\Delta(\phi^s,\Pi) :=\Delta(\phi^{\prime},\Pi)
\]
if the pair $(\phi,s)$ corresponds to $(G^{\prime},\phi^{\prime})$. For fixed $\Pi \in \Pi_{\phi}$, the function $s \mapsto \Delta(\phi^s,\Pi)$ factors through $\mathcal{S}_{\phi}$. Thus we have the function $x \mapsto \Delta(\phi^x,\Pi)$ for $x \in \mathcal{S}_{\phi}$. In addition, this function is a $\{ \pm 1 \}$-valued character of $\mathcal{S}_{\phi}$ \cite{S3}.

\bigskip

Thus in particular, given $\phi \in \Phi(G,\zeta)$ and $s \in \overline{S}_{\phi}$ semi-simple, we can consider the linear form $f \mapsto f^{\prime}(\phi^{\prime})$, with the pair $(\phi,s)$ being corresponding to $(G^{\prime},\phi^{\prime})$. This linear form, which by construction depends only on $\phi$ and $s \in \overline{S}_{\phi}$, in fact depends only on $\phi$ and $x_s$ (where $x_s$ is as above the image of $s$ in $\mathcal{S}_{\phi}$). We will denote this linear form as $f \mapsto f^{\prime}(\phi,s)$ or $f^{\prime}(\phi^s)$, and also as $f \mapsto f^{\prime}(\phi,x)$ or $f \mapsto f^{\prime}(\phi^{x})$, i.e. for $x \in \mathcal{S}_{\phi}$:
\begin{eqnarray*}
f^{\prime}(\phi,x) &=&  f^{\prime}(\phi,s) \\
f^{\prime}(\phi^{x}) &=& f^{\prime}(\phi^{s})
\end{eqnarray*}
for any $s \in \overline{S}_{\phi}$ such that $x_s =x$.

\bigskip

Shelstad had also constructed the adjoint spectral transfer factor $\Delta(\Pi,\phi^{s})$ and established the inversion formula \cite{S3}, for $\Pi \in \Pi_{\phi}$:
\begin{eqnarray*}
\tr\Pi(f) =  \sum_{x \in \mathcal{S}_{\phi}}\Delta(\Pi,\phi^{x})f^{\prime}(\phi,x).
\end{eqnarray*}
Here on the right hand side, for the adjoint spectral transfer factor, the dependence of $\Delta(\Pi,\phi^s)$ on $s$ again factors through the image $x_s$ of $s$ in $\mathcal{S}_{\phi}$, and so for $x \in \mathcal{S}_{\phi}$, we denote by $\Delta(\Pi,\phi^x)$ the value $\Delta(\Pi,\phi^s)$, for any $s \in \overline{S}_{\phi}$ such that $x_s=x$.

\bigskip
   For our purpose, we need to construct the transfer factors $\Delta(\tau,\phi^{s})$ and
         $\Delta(\phi^{s},\tau)$, in the spirit of section 5 of \cite{A3}; again these depends only on the image $x_s$ of $s$ in $\mathcal{S}_{\phi}$, and so we similarly employ the notations $\Delta(\tau,\phi^x)$ and $\Delta(\phi^x,\tau)$. If $\phi \in \Phi_{\elll}(G,\zeta)$ is elliptic, then the transfer factors $\Delta(\tau,\phi^s)$ are defined in \cite{P}:

   \[
\Delta(\tau,\phi^{s}) =\sum_{\chi \in \widehat{R}_{\pi}} \chi(r) \Delta(\Pi^{\chi},\phi^s)
                     \]
where $\tau=(M,\pi,r) \in T_{\phi}, s \in \overline{S}_{\phi}$, and $\Pi^{\chi}$ is the irreducible component of induced representation of $\pi$ whose character corresponds to that of $\chi$ (here recall that $R_{\pi}$ is an abelian elementary $2$-group, so $\chi \in \widehat{R}_{\pi}$ takes values in $\{\pm 1 \}$). We have the spectral transfer relation ({\it c.f.} equation (5.8) in \cite{P}):
 \begin{equation}\label{inverse transfer}
\Theta(\tau,f)=\sum_{x \in \mathcal{S}_{\phi}}\Delta(\tau,\phi^{x})f^{\prime}(\phi,x).
 \end{equation}

    We also define the adjoint transfer factor (still assuming $\phi \in \Phi_{\elll}(G,\zeta)$ being elliptic, and $\tau=(M,\pi,r) \in T_{\phi}$ as before):
    \[\Delta(\phi^{s},\tau)=\sum_{\chi\in\widehat{R}_{\pi}}\frac{1}{|R_{\pi}|}\chi(r) \Delta(\phi^s,\Pi^{\chi}). \]
    We then have the spectral transfer formula, for $s \in \overline{S}_{\phi}$ ({\it c.f.} equation (5.9) in \cite{P}):
    \begin{equation*}\label{inverse 2}
    f^{\prime}(\phi,s)=\sum_{\tau\in T_{\phi}}\Delta(\phi^{s},\tau)\Theta(\tau,f).
    \end{equation*}
i.e.     
    \begin{eqnarray}
f^{\prime}(\phi,x)=\sum_{\tau\in T_{\phi}}\Delta(\phi^{x},\tau)\Theta(\tau,f).
\end{eqnarray}
for $x \in \mathcal{S}_{\phi}$.

\bigskip

We have the following lemma.

 \begin{lemma}\label{deeper lemma}
  Suppose that $\phi \in \Phi(G,\zeta)$ is elliptic. Let $M$ be Levi subgroup of $G$ such that $\phi$ factors through ${}^{L}M$ as a square-integrable parameter $\phi_M \in \Phi_2(M,\zeta)$ for $M$, under which we have the split short exact sequence as in (4.1):
\[\xymatrix@C=0.5cm{
     0 \longrightarrow \mathcal{S}_{\phi_{M}} \longrightarrow  \mathcal{S}_{\phi} \longrightarrow R_{\phi}\longrightarrow 0 }.\]

Given $x \in \mathcal{S}_{\phi}$, write $x$ with respect to the above split exact sequence as $x =x_{M} \cdot  r^{\prime}$, where $x_M \in \mathcal{S}_{\phi_M}$, and $r^{\prime} \in R_{\phi}$. Then for $\tau =(M,\pi,r) \in T_{\phi}$, we have $\Delta(\phi^{x},\tau) =0$ unless $r=r^{\prime}$, in which case we have:
\[
\Delta(\phi^{x},\tau) = \Delta(\phi_M^{x_{M}}  ,\pi)
\]

  \begin{proof}
  This is a direct computation, using the properties of spectral transfer factors \cite{S3}:
     \begin{eqnarray}
      & &    \Delta(\phi^{x},\tau) = \sum_{\chi\in\widehat{R}_{\pi}}\frac{1}{|R_{\pi}|}\chi(r)  \Delta(\phi^x,\Pi^{\chi}) \\
                           &=& \sum_{\chi\in\widehat{R}_{\pi}}\frac{1}{|R_{\pi}|}\chi(r)  \chi(r^{\prime}) \Delta( \phi^{x_M}  , \Pi^{\chi}) \nonumber\\
                           &=& \delta(r,r^{\prime}) \cdot   \Delta(\phi_M^{x_M},\pi) \nonumber
                            \end{eqnarray}            where $\delta(\cdot,\cdot)$ is the Kronecker delta function. The lemma follows.
       \end{proof}
 \end{lemma}

\begin{remark}
This property of spectral transfer factor is also studied in Chapters 2 and 6 of \cite{A9}, in the form of local intertwining relations.  
\end{remark}

We now define the spectral transfer factors for general $\phi \in \Phi(G,\zeta)$, by reducing to the elliptic case. If $\phi\in\Phi(G,\zeta)$, then there exists a Levi subgroup $\widetilde{M}$ (unique up to conjugacy), that is maximal with respect to the property that $\phi$ factors through ${}^{L} \widetilde{M}$ as an elliptic parameter $\phi^{\widetilde{M}}$ for $\widetilde{M}$. With respect to the embedding ${}^{L}\widetilde{M}\hookrightarrow {}^{L}G$, the mapping \[\mathcal{S}^{}_{\phi^{\widetilde{M}}}\longrightarrow \mathcal{S}^{}_{\phi}\] is an isomorphism ({\it c.f.} proof of Lemma 4.1).
\bigskip

If $\tau=(M,\pi,r)\in T_{\phi}$, then we denote by $\tau^{\widetilde{M}}=(M,\pi,r^{\widetilde{M}})\in T_{\phi^{\widetilde{M}}}$ the corresponding element of $T_{\phi^{\widetilde{M}}}$. Here $r^{\widetilde{M}}$ corresponds to $r$ under the canonical isomorphism $R_{\phi^{\widetilde{M}}}=R_{\phi}$ ({\it c.f.} proof of Lemma 4.1).

\bigskip

We define the spectral transfer factor:
  \[\Delta(\tau,\phi^{x})=\Delta(\tau^{\widetilde{M}},(\phi^{\widetilde{M}})^{x})\]
where on the right hand side, we still denote by $x$ the element in $\mathcal{S}_{\phi^{\widetilde{M}}}$ that corresponds to $x \in \mathcal{S}_{\phi}$ under the above isomorphism between $\mathcal{S}_{\phi^{\widetilde{M}}}$ and $\mathcal{S}_{\phi}$.

 \bigskip
We have, with $f_{\widetilde{M}}$ being the descent of $f$ to $\widetilde{M}$, the identity:
  \[\Theta(\tau,f)=\Theta(\tau^{\widetilde{M}},f_{\widetilde{M}}).\]
  Similarly we have, for $x \in \mathcal{S}_{\phi}$: \[f^{\prime}(\phi,x)=(f_{\widetilde{M}})^{\prime}(\phi^{\widetilde{M}} ,x).\]

  It follows that we again have the transfer relation (generalization of (5.1)):
  \begin{eqnarray}
  \Theta(\tau,f)=\sum_{x \in\mathcal{S}_{\phi}}\Delta(\tau,\phi^{x})f^{\prime}(\phi,x).
  \end{eqnarray}

 We can similarly define the adjoint transfer factor $\Delta(\phi^{x},\tau)$, and we have the inverse transfer relation, for $x \in \mathcal{S}_{\phi}$ (generalization of (5.2))):
  \begin{eqnarray}
  f^{\prime}(\phi,x)=\sum_{\tau \in T_{\phi}} \Delta(\phi^{x},\tau)\Theta(\tau,f).
  \end{eqnarray}

 \bigskip

 We also have the following properties about the transfer factors:

 \begin{lemma}
 Let $\phi \in \Phi(G,\zeta)$ is a bounded Langlands parameter. Then we have:
 
 \begin{itemize}
 \item  $\Delta(\tau,\phi^{x})=\frac{|R_{\phi}|}{|\mathcal{S}_{\phi}|}    \overline{ \Delta(\phi^{x},\tau)}  = \frac{|R_{\phi}|}{|\mathcal{S}_{\phi}|}     \Delta(\phi^{x},\tau)$ for $\tau \in T_{\phi}$ and $x \in \mathcal{S}_{\phi}$.
 \item  we have the adjoint relations
          \[\sum_{\tau\in T_{\phi}}\Delta(\phi^{x_{1}},\tau)\Delta(\tau,\phi^{x_{2}})=\delta(x_{1},x_{2}), \mbox{ for } x_1,x_2 \in \mathcal{S}_{\phi},\]
          \[\sum_{x \in\mathcal{S}_{\phi}}\Delta(\tau_{1},\phi^{x})\Delta(\phi^{x},\tau_{2})=\delta(\tau_{1},\tau_{2}), \mbox{ for } \tau_1,\tau_2 \in T_{\phi}.\]
 \end{itemize}
 Where $\delta(\cdot,\cdot)$ is the Kronecker delta function.
 \end{lemma}
  This follows from the case where $\phi$ is an elliptic parameter, which is established in Proposition $5.2$ of \cite{P}.

\bigskip  
\bigskip

Finally to complete the discussion of this section, for $\tau \in T(G,\zeta), \phi \in \Phi(G,\zeta)$ and $x \in \mathcal{S}_{\phi}$, we define $\Delta(\tau,\phi^x)$ and $\Delta(\phi^x,\tau)$ to be zero if $\tau \notin T_{\phi}$.  

  \section{Stabilization of the local trace formula}

     To summarize the discussion of the previous section, we have firstly, from equation (5.4) and (5.5), the following:
     
     \begin{theorem}
     If $f\in \mathcal{H}(G(\R),\zeta)$, then for $\tau \in T(G,\zeta)$, we have:
     \begin{eqnarray} \label{general transfer}
          \Theta(\tau,f)=\sum_{(\phi,x)}\Delta(\tau,\phi^{x})f^{\prime}(\phi,x)    
      \end{eqnarray}
      where the summation is over $\phi \in\Phi(G,\zeta)$ and $x \in \mathcal{S}_{\phi}$.
            
            \bigskip
     Conversely, for $\phi \in \Phi(G,\zeta), x \in \mathcal{S}_{\phi}$, we have:
     \begin{eqnarray}\label{general inverse}
          f^{\prime}(\phi,x)=\sum_{\tau\in T(G,\zeta)}\Delta(\phi^{x},\tau)\Theta(\tau,f).     \end{eqnarray}
     \end{theorem}
  
  The following also follows easily:
  
  \begin{lemma}\label{general deeper lemma}
  The statement of Lemma 5.3 holds for general $\phi \in \Phi(G,\zeta)$.
    \end{lemma}
 
 \bigskip
 
Now for $\phi \in \Phi(G,\zeta)$, define $T_{\phi,\disc} := T_{\phi} \cap T_{\disc}(G,\zeta)$. We have $T_{\phi,\disc}$ is non-empty if and only if $\phi \in \Phi_{\disc}(G,\zeta)$ ({\it c.f.} Section 5 of \cite{S1}). Now recall from Lemma 4.1, one has the bijection $\iota: T_{\phi} \rightarrow \mathcal{S}_{\phi}$ that respects the projection to $R_{\phi}$. Thus for $\phi \in \Phi_{\disc}(G,\zeta)$, put $\mathcal{S}_{\phi,\disc}$ to be the image of $T_{\phi,\disc}$ under the bijection $\iota: T_{\phi} \rightarrow \mathcal{S}_{\phi}$. We define:
\[\Phi^{S}_{\disc}(G,\zeta)=\{(\phi,x):\phi\in\Phi_{\disc}(G,\zeta),x \in\mathcal{S}_{\phi,\disc}\}.\]

\bigskip
  Thus for $\tau\in T_{\disc}(G,\zeta)$, on applying equation (6.1) and the above discussion, we have the transfer relation:
  \begin{equation}
  \Theta(\tau,f)=\sum_{(\phi,x)\in\Phi^{S}_{\disc}(G,\zeta)}\Delta(\tau,\phi^{x})f^{\prime}(\phi,x).
  \end{equation}
  Similarly for $\phi \in \Phi_{\disc}(G,\zeta)$ and $x \in \mathcal{S}_{\phi,\disc}$, we have, on applying equation (6.2), the inversion formula:
  \begin{eqnarray}
  f^{\prime}(\phi,x)=\sum_{\tau\in T_{\disc}(G,\zeta)}\Delta(\phi^{x},\tau)\Theta(\tau,f).
  \end{eqnarray}

\bigskip

In the following, we also denote a pair $(\phi,x) \in \Phi^S_{\disc}(G,\zeta)$ as $\phi^x$. 

\bigskip

The linear space $i\mathfrak{a}^{\ast}_{G,Z}$ acts on $\Phi_{\disc}(G,\zeta)$ through twisting: for $\phi \in \Phi_{\disc}(G,\zeta)$ and $\lambda \in i\mathfrak{a}^{\ast}_{G,Z}$
  \[\phi_{\lambda}(\omega):=\phi(\omega)|\omega|^{\lambda}, \quad \omega\in W_{\R}\]
(here the element $|\omega|^{\lambda}\in Z(\widehat{G})^{\Gamma_{\R}}$ is defined via the usual Nakayama-Tate duality). This induces the action of $i\mathfrak{a}^{\ast}_{G,Z}$ on $\Phi^S_{\disc}(G,\zeta)$, through twisting on $\Phi_{\disc}(G,\zeta)$: $(\phi^x)_{\lambda} :=(\phi_{\lambda})^x$; here we are identifying $\overline{S}_{\phi}$ and $\overline{S}_{\phi_{\lambda}}$, hence the identification for $\mathcal{S}_{\phi}$ and $\mathcal{S}_{\phi_{\lambda}}$, similarly identification for $\mathcal{S}_{\phi,\disc}$ and $\mathcal{S}_{\phi_{\lambda},\disc}$.

\bigskip

We define a measure on $\Phi_{\disc}(G,\zeta)$ and $\Phi^{S}_{\disc}(G,\zeta)$ by setting:
\[
\int_{\Phi_{\disc}(G,\zeta)}\beta_1(\phi)d\phi=\sum_{\phi \in \Phi_{\disc}(G,\zeta)/ i\mathfrak{a}^{\ast}_{G,Z}}\int_{i\mathfrak{a}^{\ast}_{G,Z}}\beta_1(\phi_{\lambda})d\lambda,
\]
   \[\int_{\Phi^{S}_{\disc}(G,\zeta)}\beta_2(\phi^{x})d\phi^{x}=\sum_{(\phi,x)\in \Phi^{S}_{\disc}(G,\zeta)/ i\mathfrak{a}^{\ast}_{G,Z}}\int_{i\mathfrak{a}^{\ast}_{G,Z}}\beta_2(\phi^{x}_{\lambda})d\lambda,\]
   for any $\beta_1\in C_c(\Phi_{\disc}(G,\zeta)), \beta_2\in C_c(\Phi^{S}_{\disc}(G,\zeta))$. We have the following lemma, similar to Lemma 5.3 of \cite{A3}:

   \begin{lemma} \label{change integrable}
   Suppose that $\alpha\in C_c(T_{\disc}(G,\zeta))$, and $\beta\in C_c(\Phi^{S}_{\disc}(G,\zeta))$. Then we have
   \begin{equation} \label{chang measure}
   \int_{T_{\disc}(G,\zeta)}\sum_{(\phi,x)\in\Phi^{S}_{\disc}(G,\zeta)}\beta(\phi^{x})\Delta(\phi^{x},\tau)\alpha(\tau)d\tau,
   \end{equation}

   \begin{equation*}\label{change measure 2}
   =\int_{\Phi^{S}_{\disc}(G,\zeta)}\sum_{\tau\in T_{\disc}(G,\zeta)}\beta(\phi^{x})\Delta(\phi^{x},\tau)\alpha(\tau)d\phi^{x}.
   \end{equation*}
\end{lemma}
\begin{proof}
   First note that for given a $\tau$, the first inner summation is a finite sum, by the definition of the transfer factors; similarly for a given $\phi^x$, the second summation is also finite. So the identity makes sense. From the definition of measure for $T_{\disc}(G,\zeta)$, we can write the left hand side of (6.5) as:

   \begin{equation}\label{change proof 1}
   \sum_{\tau \in T_{\disc}(G,\zeta)/ i\mathfrak{a}^{\ast}_{G,Z}}\sum_{\phi^{x}\in\Phi^{S}_{\disc}(G,\zeta)/ i\mathfrak{a}^{\ast}_{G,Z}} \sum_{\mu \in i \mathfrak{a}^{\ast}_{G,Z}} \int_{i\mathfrak{a}^{\ast}_{G,Z}} \beta(\phi^{x}_{\mu})\Delta(\phi^{x}_{\mu},\tau_{\lambda})\alpha(\tau_{\lambda}) d\lambda
\end{equation}

\bigskip

For a given representative $\tau \in T_{\disc}(G,\zeta)/ i \mathfrak{a}^{\ast}_{G,Z}$ in the outer sum, we may choose the representative $\phi^x \in \Phi^S(G,\zeta)/ i \mathfrak{a}^{\ast}_{G,Z}$ in the inner sum, that has the same central character on $A_G$. 
 
\bigskip
 
Thus in the sum over $\mu$, we see that if $\lambda\neq\mu$, then $\phi^{x}_{\mu}$ and $\tau_{\lambda}$ are not related, and so by the definition of transfer factors, we have $\Delta(\phi^{x}_{\mu},\tau_{\lambda})=0$. So we can write (6.6) as:

\begin{equation}\label{change proof 2}
 \sum_{(\tau,\phi^x)} \int_{i\mathfrak{a}^{\ast}_{G,Z}}\beta(\phi^{x}_{\lambda})\Delta(\phi^{x}_{\lambda},\tau_{\lambda})\alpha(\tau_{\lambda})d\lambda.
\end{equation}
where double sum of (6.7) is a sum over the subset \[(\tau,\phi^{x})\in (T_{\disc}(G,\zeta)\times\Phi^{S}_{\disc}(G,\zeta))/ i\mathfrak{a}^{\ast}_{G,Z}\]
consisting of those pairs that have the same central character on $A_G$.

\bigskip

In a parallel way, we can treat the right hand side of (6.5) along similar lines, using the definition of measure for $\Phi^{S}_{\disc}(G,\zeta)$. On noting the symmetry of the sum-integral in (6.7), we thus conclude that the left hand side and the right hand side of (6.5) are equal. 
\end{proof}

\bigskip

We can now begin the stabilization of the spectral side of the local trace formula:

 \begin{theorem}
 If $f=f_{1}\times\bar{f_{2}}$, $f_{i}\in \mathcal{H}(G(\R),\zeta)$ for $i=1,2$, then we have:

  \begin{equation}
  I^G_{\disc}(f)=\int_{T_{\disc}(G,\zeta)}i^{G}(\tau) \Theta(\tau,f_1)  \overline{\Theta(\tau,f_2)} |R_{\pi}|^{-1} d\tau
  \end{equation}
             \[ =\int_{\Phi^{S}_{\disc}(G,\zeta)}\frac{1}{|\mathcal{S}_{\phi}|}i_{\phi}(x)f^{\prime}_{1}(\phi,x)\overline{f^{\prime}_{2}(\phi,x)}d\phi^{x}\]

 \begin{proof}
Firstly, applying the spectral transfer, as given in equation (6.3), to the term $\Theta(\tau,f_1)$ in (6.8), we see that $I^G_{\disc}(f)$ equals
\begin{eqnarray}\int_{T_{\disc}(G,\zeta)}i^{G}(\tau)|R_{\pi}|^{-1}\sum_{(\phi,x)\in \Phi^{S}_{\disc}(G,\zeta)}\Delta(\tau,\phi^{x})f^{\prime}_{1}(\phi,x)\overline{\Theta(\tau,f_2)} d\tau \end{eqnarray}
\[
= \int_{T_{\disc}(G,\zeta)}i^{G}(\tau)|R_{\pi}|^{-1}\sum_{(\phi,x)\in\Phi^{S}_{\disc}(G,\zeta)}\frac{|R_{\phi}|}{|\mathcal{S}_{\phi}|}  f^{\prime}_{1}(\phi,x)\overline{\Delta(\phi^{x},\tau) \Theta(\tau,f_2)}d\tau.
\]
with the last equality follows from the first part of Lemma 5.5. Next, in order that $\Delta(\phi^x,\tau) \neq 0$, we must have $\tau \in T_{\phi}$ and hence $|R_{\pi}|=|R_{\phi}|$ in the integrand. Thus we see that the right hand side of (6.9) can be written as:
\[
 \int_{T_{\disc}(G,\zeta)}\sum_{(\phi,x)\in\Phi^{S}_{\disc}(G,\zeta)}\frac{i^G(\tau)}{|\mathcal{S}_{\phi}|}  f^{\prime}_{1}(\phi,x)\overline{\Delta(\phi^{x},\tau) \Theta(\tau,f_2)}d\tau.
\]
By Lemma 6.3, this is equal to:

\begin{equation}\label{stable2}
 \int_{\Phi^{S}_{\disc}(G,\zeta)}\frac{1}{|\mathcal{S}_{\phi}|}f^{\prime}_{1}(\phi,x)\sum_{\tau\in T_{\disc}(G,\zeta)}i^{G}(\tau)\overline{\Delta(\phi^{x},\tau) \Theta(\tau,f_2)} d\phi^x.
 \end{equation}
 Now, denote by $P: \mathcal{S}_{\phi}\rightarrow R_{\phi}$ for the surjective map from $\mathcal{S}_{\phi}$ to $R_{\phi}$ in the split short exact sequence (4.1). Given $\tau = (M,\pi,r) \in T_{\phi}$, and  $\phi \in \Phi_{\disc}(G,\zeta)$, by Lemma 6.2, for $x \in \mathcal{S}_{\phi,\disc}$, if $P(x)\neq r,$ then $\Delta(\phi^{x},\tau)=0$. On the other hand, if $\Delta(\phi^{x},\tau)\neq 0$, then we must have $P(x)=r$, i.e. $\iota(\tau)$ and $x$ have the same image in $R_{\phi}$ under the map $P$. Thus we have:
  \[i^{G}(\tau)=i_{\phi}(\iota(\tau)) =i_{\phi}(x).\]
({\it c.f.} the discussion after the statement of Theorem 4.2, near the end of Section 4). 

\bigskip

Thus we can write (6.10) as
 \[\int_{\Phi^{S}_{\disc}(G,\zeta)}\frac{1}{|\mathcal{S}_{\phi}|}i_{\phi}(x)f^{\prime}_{1}(\phi,x)\sum_{\tau\in T_{\disc}(G,\zeta)}\overline{\Delta(\phi^{x},\tau) \Theta(\tau,f_2)} d\phi^x.\]
Applying equation (6.4), we obtain the required formula:
 \begin{equation}\label{stable lemma}
I^G_{\disc}(f)= \int_{\Phi^{S}_{\disc}(G,\zeta)}\frac{1}{|\mathcal{S}_{\phi}|}i_{\phi}(x)f^{\prime}_{1}(\phi,x)\overline{f^{\prime}_{2}(\phi,x)}d\phi^{x}.
 \end{equation}
 \end{proof}
 \end{theorem}

 We set $i_{\phi}(x)=0$, if $x \notin  \mathcal{S}_{\phi,\disc}$. Then we have:
  \begin{eqnarray*}
  I^G_{\disc}(f)=\int_{\Phi_{\disc}(G,\zeta)}\sum_{x\in\mathcal{S}_{\phi}}|\mathcal{S}_{\phi}|^{-1}i_{\phi}(x)f^{\prime}_{1}(\phi,x)\overline{f^{\prime}_2(\phi,x)}d\phi.
  \end{eqnarray*}

Recall equation (4.4) from section 4, we write:
\[i_{\phi}(x)=\sum_{s\in\mathcal{E}_{\phi,\elll}(x)}|\pi_{0}(\overline{S}_{\phi,s})|^{-1}\sigma(\overline{S}_{\phi,s}^{\circ})\]

\bigskip
Thus we obtain
 \begin{equation*}
  I^G_{\disc}(f)=\int_{\Phi_{\disc}(G,\zeta)}\sum_{x\in \mathcal{S}_{\phi}}|\mathcal{S}_{\phi}|^{-1}\sum_{s\in\mathcal{E}_{\phi,\elll}(x)}|\pi_{0}(\overline{S}_{\phi,s})|^{-1}\sigma(\overline{S}_{\phi,s}^{\circ})f^{\prime}_{1}(\phi,s)\overline{f^{\prime}_{2}(\phi,s)}d\phi \end{equation*}
in other words,  
\begin{eqnarray}
 I^G_{\disc}(f)=\int_{\Phi_{\disc}(G,\zeta)}\sum_{s \in \mathcal{E}_{\phi,\elll}}|\mathcal{S}_{\phi}|^{-1}|\pi_{0}(\overline{S}_{\phi,s})|^{-1}\sigma(\overline{S}_{\phi,s}^{\circ})f^{\prime}_{1}(\phi,s)\overline{f^{\prime}_{2}(\phi,s)}d\phi.
\end{eqnarray}

\noindent This is Theorem 1.1 as stated in the Introduction. In the next section, we are going to express the right hand side of (6.12) in the form of an endoscopic local trace formula. 

 \section{Spectral side of endoscopic local trace formula}

To obtain Theorem 1.2 as stated in the Introduction, we need to rewrite the right hand side of (6.12), in terms of the endoscopic data of $G$. To do this we need to make precise the correspondence $(G^{\prime},\phi^{\prime}) \longleftrightarrow (\phi,s)$.

    \bigskip
    
    Denote by $E_{\elll}(G)$ for the set of elliptic endoscopic data of $G$. The set $\widehat{G} \bbslash E_{\elll}(G)$ of $\widehat{G}$-orbits in $E_{\elll}(G)$ is then equal to the set $\mathcal{E}_{\elll}(G)$ of equivalence classes of elliptic endoscopic data of $G$.
   
     \bigskip
Write $F_{\disc}(G,\zeta)$ for the set of bounded Langlands parameters $\phi$ of $G$ that have character $\zeta$ with respect to $Z$ ({\it not} being regarded as up to $\widehat{G}$-equivalence), and such that $Z(\overline{S}_{\phi})$ is finite. Similarly for $G^{\prime} \in E_{\elll}(G)$, write $F_{s-\disc}(G^{\prime},\zeta)$ for the set of bounded Langlands parameters $\phi^{\prime}$ of $G^{\prime}$ that have character $\zeta$ with respect to $Z$ ({\it not} being regarded as up to $\widehat{G^{\prime}}$-equivalence), and such that $Z(\overline{S}_{\phi}^{\circ})$ is finite.

\bigskip
 The set $\widehat{G} \bbslash F_{\disc}(G,\zeta)$ of $\widehat{G}$-orbits in $F_{\disc}(G,\zeta)$ is equal to $\Phi_{\disc}(G,\zeta)$. Similarly, for $G^{\prime} \in E_{\elll}(G)$, the set $\widehat{G^{\prime}} \bbslash F_{s-\disc}(G^{\prime},\zeta)$ of $\widehat{G^{\prime}}$-orbits in $F_{s-\disc}(G^{\prime},\zeta)$ is equal to $\Phi_{s-\disc}(G^{\prime},\zeta)$.

\bigskip

    We then define:
     \[ X_{\disc}(G,\zeta):=\{(\phi,s):\phi\in F_{\disc}(G,\zeta),s\in\overline{S}_{\phi,\elll}\},\]
    and
    \[Y_{\disc}(G,\zeta):=\{(G^{\prime},\phi^{\prime}):G^{\prime}\in E_{\elll}(G),\phi^{\prime}\in F_{s-\disc}(G^{\prime},\zeta)\}.\]
    
    \bigskip
    
    The group $\widehat{G}$ acts on $X_{\disc}(G,\zeta)$ and $Y_{\disc}(G,\zeta)$ by conjugation. Denote by $\widehat{G} \bbslash X_{\disc}(G,\zeta)$ and $\widehat{G} \bbslash Y_{\disc}(G,\zeta)$ for the set of $\widehat{G}$-orbits.

\bigskip

As in the previous section, the linear space $i \mathfrak{a}^{\ast}_{G,Z}$ acts on $\widehat{G} \bbslash X_{\disc}(G,\zeta)$ and $\widehat{G} \bbslash Y_{\disc}(G,\zeta)$ by twisting, and we define measures on $\widehat{G} \bbslash X_{\disc}(G,\zeta)$ and $\widehat{G} \bbslash Y_{\disc}(G,\zeta)$ by the same type of formula that define the measure on $\Phi_{\disc}(G,\zeta) = \widehat{G} \bbslash F_{\disc}(G,\zeta)$, i.e.
\[
\int_{ \widehat{G} \bbslash X_{\disc}(G,\zeta)} =\sum_{  \widehat{G} \bbslash X_{\disc}(G,\zeta)/ i\mathfrak{a}^{\ast}_{G,Z}}\int_{i\mathfrak{a}^{\ast}_{G,Z}}
\]
   \[\int_{\widehat{G}  \bbslash Y_{\disc}(G,\zeta)}=\sum_{ \widehat{G} \bbslash  Y_{\disc}(G,\zeta)/ i\mathfrak{a}^{\ast}_{G,Z}}\int_{i\mathfrak{a}^{\ast}_{G,Z}}\]
    
\bigskip
  The correspondence $(\phi,s) \longleftrightarrow (G^{\prime},\phi^{\prime})$ induces the bijections: 
  \begin{eqnarray}
  \widehat{G} \bbslash X_{\disc}(G,\zeta) & \longleftrightarrow &\widehat{G} \bbslash Y_{\disc}(G,\zeta) \\
  \widehat{G} \bbslash X_{\disc}(G,\zeta)  / i \mathfrak{a}^*_{G,Z}  & \longleftrightarrow & \widehat{G} \bbslash Y_{\disc}(G,\zeta) / i \mathfrak{a}^*_{G,Z} \nonumber   \end{eqnarray}
which is immediately seen to be measure preserving, and which is the focal point for the transformation of the right hand side of the expression (6.12). The argument we give below is parallel to that of Section 4.4 of \cite{A9}.
   
   \bigskip
   
   The first step is to change the double sum-integral on the right hand side of (6.12) to an integral over $\widehat{G} \bbslash X_{\disc}(G,\zeta)$, using that $\Phi_{\disc}(G,\zeta) = \widehat{G} \bbslash F_{\disc}(G,\zeta)$, and also that the integrand is $\widehat{G}$-invariant. Given $\phi \in F_{\disc}(G,\zeta)$, the stabilizer of $\phi$ in $\widehat{G}$ is the centralizer $S_{\phi}$. Now the sum occurring in the integrand on the right hand side of (6.12), is over $\mathcal{E}_{\phi,\elll}=\overline{S}_{\phi}^{\circ} \bbslash \overline{S}_{\phi,\elll}$, the set of orbits in $\overline{S}_{\phi,\elll}$ under the conjugation action by the identity component $\overline{S}_{\phi}^{\circ}$ of $\overline{S}_{\phi}$. On the other hand, given $s \in \overline{S}_{\phi,\elll}$, the set of orbits under the conjugation action of $S_{\phi}$, or equivalently by $\overline{S}_{\phi}$, is bijective with the quotient of $\overline{S}_{\phi}$ by the subgroup
    \[\overline{S}^{+}_{\phi,s}=\Cent(s,\overline{S}_{\phi}).\]

   However, the $\overline{S}^{\circ}_{\phi}$-orbit of $s$ is bijective with the quotient of $\overline{S}^{\circ}_{\phi}$ by the subgroup

    \[\overline{S}_{\phi,s}=\Cent(s,\overline{S}_{\phi}^{\circ}).\]

    We can therefore rewrite the sum-integral on the right hand side of (6.12), as an integral over $\widehat{G} \bbslash X_{\disc}(G,\zeta)$, if we multiply the summand on the right hand side of (6.12) by the number:
     \[|\overline{S}_{\phi}^{\circ}/\overline{S}_{\phi,s}|^{-1} |\overline{S}_{\phi}/\overline{S}^+_{\phi,s}|  =  |\overline{S}_{\phi,s}^{+}/\overline{S}_{\phi,s}|^{-1}|\overline{S}_{\phi}/\overline{S}_{\phi}^{\circ}|,\]
     which is to say the number:
     
   \begin{align} \label{coefficint change 1}
   |\overline{S}_{\phi,s}^{+}/\overline{S}_{\phi,s}|^{-1}|\mathcal{S}_{\phi}|.
   \end{align}

\bigskip      
      The second step is to use the bijection (7.1) and write the right hand side of (6.12) as an integral over $\widehat{G} \bbslash Y_{\disc}(G,\zeta)$. Recall firstly that:
       \[\mathcal{E}_{\elll}(G)= \widehat{G} \bbslash E_{\elll}(G).\]
       
       Now, the stabilizer of a given $G^{\prime} \in E_{\elll}(G)$ in $\widehat{G}$ is the group $\Aut_{G}(G^{\prime})$. This means that the integral over $\widehat{G} \bbslash Y_{\disc}(G,\zeta)$ could be written as a double sum-integral:
\[
 \sum_{G^{\prime} \in \mathcal{E}_{\elll}(G)} \int_{\Aut_G(G^{\prime}) \bbslash F_{s-\disc}(G^{\prime},\zeta)}
\]       

Now the integral over $\phi^{\prime} \in \Aut_G(G^{\prime}) \bbslash F_{s-\disc}(G^{\prime},\zeta)$, could be replaced by the integral over $G^{\prime} \bbslash F_{s-\disc}(G^{\prime},\zeta)= \Phi_{s-\disc}(G^{\prime},\zeta)$, so long as we multiply the integrand by the number:

    \begin{align} \label{coefficint change 2}
   |\Out_{G}(G^{\prime})|^{-1}|\Out_{G}(G^{\prime},\phi^{\prime})|
   \end{align}
where $\Out_{G}(G^{\prime},\phi^{\prime})$ is the stabilizer of $\phi^{\prime}$ in $\Out_{G}(G^{\prime})$.
   
   \bigskip
   We have now established that the double sum-integral on the right hand side of (6.12), can be replaced by the double sum-integral:
   \[
   \sum_{G^{\prime} \in \mathcal{E}_{\elll}(G)} \int_{\Phi_{s-\disc}(G^{\prime},\zeta)}
   \]
   provided that the summand is multiplied by the product of the two numbers (\ref{coefficint change 1}) and (\ref{coefficint change 2}).
     Finally, the coefficient occurring in the summand on the right hand side of (6.12) is:
        \begin{align} \label{coefficint change 3}
   |\mathcal{S}_{\phi}|^{-1}|\pi_{0}(\overline{S}_{\phi,s})|^{-1}\sigma(\overline{S}^{\circ}_{\phi,s}).
   \end{align}
   
   Thus we have:
   \[
   I^G_{\disc}(f) = \sum_{G^{\prime} \in \mathcal{E}_{\elll}(G)} \int_{\Phi_{s-\disc}(G^{\prime},\zeta)} (7.2) \cdot (7.3) \cdot (7.4) \cdot  f_1^{\prime}(\phi,s) \overline{f_2^{\prime}(\phi,s)}  d \phi^{\prime} \cdot
   \]
   
   \bigskip
   
We thus need to express the product of (7.2), (7.3), and (7.4), in terms of the pair $(G^{\prime},\phi^{\prime})$.
\bigskip
   
 The product of (\ref{coefficint change 1}), (\ref{coefficint change 2}) and (\ref{coefficint change 3}) is equal to the product of:   \[|\Out_{G}(G^{\prime})|^{-1}\] 
   and

   \begin{align} \label{coefficint change 4}
   |\Out_{G}(G^{\prime},\phi^{\prime})||\overline{S}_{\phi,s}^{+}/\overline{S}_{\phi,s}|^{-1}|\pi_{0}(\overline{S}_{\phi,s})|^{-1}\sigma(\overline{S}_{\phi,s}^{\circ}).
   \end{align}

\bigskip
Now under the correspondence $(G^{\prime},\phi^{\prime}) \longleftrightarrow (\phi,s)$, we have:
 \begin{align*}
 |\Out_{G}(G^{\prime},\phi^{\prime})|&=|S_{\phi,s}^{+}/S_{\phi,s}^{+}\cap\widehat{G^{\prime}}Z(\widehat{G})^{\Gamma_{\R}}| \\
                                     &=|\overline{S}_{\phi,s}^{+}/\overline{S}_{\phi,s}^{+}\cap\overline{\widehat{G^{\prime}}}|
\end{align*}
where $\overline{\widehat{G^{\prime}}}$ denote the quotient
\[\widehat{G^{\prime}}Z(\widehat{G})^{\Gamma_{\R}}/Z(\widehat{G})^{\Gamma_{\R}}\cong \widehat{G^{\prime}}/\widehat{G^{\prime}} \cap Z(\widehat{G})^{\Gamma_{\R}}.\]

\bigskip

Also \[\sigma(\overline{S}_{\phi^{\prime}}^{\circ})=\sigma((S_{\phi^{\prime}}/Z(\widehat{G^{\prime}})^{\Gamma_{\R}})^{\circ})=\sigma(\overline{S}_{\phi,s}^{\circ}/\overline{S}^{\circ}_{\phi,s}\cap\overline{Z}(\widehat{G^{\prime}})^{\Gamma_{\R}})\]
where
\[
\overline{Z}(\widehat{G^{\prime}})^{\Gamma_{\R}} =Z(\widehat{G^{\prime}})^{\Gamma_{\R}} / Z(\widehat{G})^{\Gamma_{\R}}.
\]

\bigskip

 We thus have the identity
 \begin{align} \label{coefficint change 5}
   \sigma(\overline{S}^{\circ}_{\phi^{\prime}})=\sigma(\overline{S}^{\circ}_{\phi,s})|\overline{S}_{\phi,s}^{\circ}\cap\overline{Z}(\widehat{G^{\prime}})^{\Gamma_{\R}}|
\end{align}
by Theorem \ref{coefficient endos theorem}.
 \bigskip
 
 The term
  \[|\pi_{0}(\overline{S}_{\phi,s})|^{-1}|\overline{S}_{\phi,s}^{\circ}\cap\overline{Z}(\widehat{G^{\prime}})^{\Gamma_{\R}}|^{-1}\]
   equals
   \begin{align} \label{coefficint change 6}
   |\overline{S}_{\phi,s}/\overline{S}_{\phi,s}^{\circ}\overline{Z}(\widehat{G^{\prime}})^{\Gamma_{\R}}|^{-1}|\overline{Z}(\widehat{G^{\prime}})^{\Gamma_{\R}}|^{-1}
\end{align}
as can be seen by using \[\overline{S}^{\circ}_{\phi,s}/\overline{S}_{\phi,s}^{\circ}\cap\overline{Z}(\widehat{G^{\prime}})^{\Gamma_{\R}}\cong \overline{S}_{\phi,s}^{\circ}\overline{Z}(\widehat{G^{\prime}})^{\Gamma_{\R}}/\overline{Z}(\widehat{G^{\prime}})^{\Gamma_{\R}}.\]
\bigskip

Finally we can write:
 \begin{align} \label{coefficint change 7}
   |\mathcal{S}_{\phi^{\prime}}|&=|\pi_{0}(\overline{S}_{\phi^{\prime}})|\\
                               &=|\overline{S}_{\phi,s}^{+}\cap\overline{\widehat{G^{\prime}}}/(\overline{S}_{\phi,s}^{+})^{\circ}\overline{Z}(\widehat{G^{\prime}})^{\Gamma_{\R}}| \nonumber \\
                               &=|\overline{S}_{\phi,s}^{+}\cap\overline{\widehat{G^{\prime}}}/ \overline{S}^{\circ}_{\phi,s} \overline{Z}(\widehat{G^{\prime}})^{\Gamma_{\R}}| \nonumber
\end{align}
 where the last equality follows on noting that
\begin{align} \label{coefficint change 8}
  (\overline{S}^{+}_{\phi,s})^{\circ}=\overline{S}_{\phi,s}^{\circ}.
\end{align}

\bigskip
Thus (7.5) is equal to:

\[|\overline{Z}(\widehat{G^{\prime}})^{\Gamma}|^{-1}|\mathcal{S}_{\phi^{\prime}}|^{-1}\sigma(\overline{S}_{\phi^{\prime}}^{\circ}).\]

\bigskip
Finally, under the correspondence $(G^{\prime},\phi^{\prime}) \longleftrightarrow (\phi,s)$, we have 
\[
f_1^{\prime}(\phi,s) \overline{f_2^{\prime}(\phi,s)} = f_1^{G^{\prime}}(\phi^{\prime}) \overline{ f_2^{G^{\prime}}(\phi^{\prime})}.
\]

  \bigskip
  We thus obtain the following main theorem:

 \begin{theorem}
 If $f=f_{1}\times \bar{f}_{2}$, $f_{i}\in \mathcal{H}(G(\R),\zeta),i=1,2$. Then we have
 \begin{equation}
 I^G_{\disc}(f)=\sum_{G^{\prime}\in\mathcal{E}_{\elll}(G)}\iota(G,G^{\prime})\cdot \int_{\Phi_{s-\disc}(G^{\prime},\zeta)}|\mathcal{S}_{\phi^{\prime}}|^{-1}\sigma(\overline{S}_{\phi^{\prime}}^{\circ})f_{1}^{G^{\prime}}(\phi^{\prime})\overline{f_{2}^{G^{\prime}}(\phi^{\prime})}d\phi^{\prime} \end{equation}
where
\[\iota(G,G^{\prime})=|\Out_{G}(G^{\prime})|^{-1}|\overline{Z}(\widehat{G^{\prime}})^{\Gamma_{\R}}|^{-1}.\]
 \end{theorem}
 
\bigskip
Now we define, for any quasi-split $K$-group $G$ over $\R$, the following stable distribution for $G$:
\[
S^G_{\disc}(f) :=  \int_{\Phi_{s-\disc}(G,\zeta)}|\mathcal{S}_{\phi}|^{-1}\sigma(\overline{S}_{\phi}^{\circ})f_{1}^{G}(\phi)\overline{f_{2}^{G}(\phi)}d\phi . 
\]
where as before $f=f_1 \times \bar{f}_2$, $f_1,f_2 \in \mathcal{H}(G(\R),\zeta)$. In particular one has the stable distribution $S_{\disc}^{G^{\prime}}$ for $G^{\prime}$, with $G^{\prime}$ being an endoscopic datum of $G$. 

\bigskip
Still with $f= f_1 \times \bar{f}_2$, $f_1, f_2 \in \mathcal{H}(G(\R),\zeta)$, and $G^{\prime}$ an endoscopic datum of $G$, one has functions $f_1^{\prime},f_2^{\prime} \in \mathcal{H}(G^{\prime},\zeta)$ such that the stable orbital integrals of $f_i^{\prime}$ is equal to the Langlands-Shelstad transfer $f_i^{G^{\prime}}$ ($i=1,2$) \cite{S1}. Put $f^G := f_1^{G^{\prime}} \times \bar{f}_2^{G^{\prime}}$ and $f^{\prime} := f_1^{\prime} \times \bar{f}^{\prime}_2$. Then since $S_{\disc}^{G^{\prime}}$ is stable for $G^{\prime}$, one has that the value $S_{\disc}^{G^{\prime}}(f^{\prime})$ depends only on the stable orbital integral of $f^{\prime}$, i.e. only on the Langlands-Shelstad transfer $f^{G^{\prime}}$. We thus denote the value $S_{\disc}^{G^{\prime}}(f^{\prime})$ as $\widehat{S}_{\disc}^{G^{\prime}}(f^{G^{\prime}})$.

\bigskip
With these notations we have:
\[
\widehat{S}^{G^{\prime}}_{\disc}(f^{G^{\prime}}) =  \int_{\Phi_{s-\disc}(G^{\prime},\zeta)}|\mathcal{S}_{\phi^{\prime}}|^{-1}\sigma(\overline{S}_{\phi^{\prime}}^{\circ})f_{1}^{G^{\prime}}(\phi^{\prime})\overline{f_{2}^{G^{\prime}}(\phi^{\prime})}d\phi^{\prime} \]
hence in the context of Theorem 7.1, we can write equation (7.10) in the following form:
\begin{equation}
I^G_{\disc}(f)=\sum_{G^{\prime}\in\mathcal{E}_{\elll}(G)}\iota(G,G^{\prime})\widehat{S}^{G^{\prime}}_{\disc}(f^{G^{\prime}}).
\end{equation}

\section{Spectral side of stable local trace formula}

With Theorem 7.1 in hand, we can now obtain the explicit formula for the spectral side of the stable local trace formula. As before $G$ is any $K$-group over $\R$ that is quasi-split, with central data $(Z,\zeta)$. 

\bigskip

In \cite{A6}, one has following stable distribution $S^G$ for $G$, which is the stable version of the geometric side of the local trace formula $I^G$, and is defined as follows. For $f=f_{1}\times\bar{f}_{2}$, with $f_1,f_2 \in \mathcal{H}(G(\R),\zeta)$:
\begin{equation}
 S^{G}(f)=\sum_{M\in\mathcal{L}}|W^{M}_{0}||W^{G}_{0}|^{-1}(-1)^{\dimm(A_{M}/ A_{G})}\int_{\Delta_{G-\reg,\elll}(M,V,\zeta)}n(\delta)^{-1}S^G_{M}(\delta,f)d\delta
\end{equation}
{\it c.f.} equation (10.11) of \cite{A6}. Here $n(\delta)$ is the order of the group $\mathcal{K}_{\delta}$ as defined on p. 509 of \cite{A3}, and $\Delta_{G-\reg,\elll}(M,V,\zeta)$ is the stable version of $\Gamma_{G-\reg,\elll}(M,V,\zeta)$, similarly $S^G_M(\delta,f)$ is the stable version of $I^G_M(\gamma,f)$.

\bigskip
In particular one has the stable distribution $S^{G^{\prime}}$ for $G^{\prime}$, with $G^{\prime}$ being an endoscopic datum of $G$. It is then shown in \cite{A6} that, the geometric side of the local trace formula $I^G(f)$ for $G$ (where $f = f_1 \times \bar{f}_2$, $f_1,f_2 \in \mathcal{H}(G(\R),\zeta)$ as above) satisfies the following endoscopic decomposition:
\begin{equation}
I^G(f)=\sum_{G^{\prime}\in\mathcal{E}_{\elll}(G)}\iota(G,G^{\prime})\widehat{S}^{G^{\prime}}(f^{G^{\prime}})
\end{equation}
where as before $f^{G^{\prime}}$ is the Langlands-Shelstad transfer of $f$ to $G^{\prime}$, {\it c.f.} equation (10.16) of \cite{A6}. Here the meaning of $\widehat{S}^{G^{\prime}}(f^{G^{\prime}})$ is similar to that of concerning $\widehat{S}^{G^{\prime}}_{\disc}(f^{G^{\prime}})$, {\it c.f.} the discussion near the end of Section 7.

\bigskip
We remark that, for the archimedean case of the local trace formula, the geometric transfer identities that are needed in \cite{A6} to establish the endoscopic decomposition (8.1)-(8.2), were established directly in \cite{A8}, and so it is independent of global arguments (in the non-archimedean case, the endoscopic decomposition of the local trace formula was established in \cite{A6} using global arguments, and so the fundamental lemma is needed in the non-archimedean case). We also remark that, when one of the components of $f$ is cuspidal, then the arguments for the endoscopic decomposition (8.1)-(8.2) were already carried out in \cite{A4}, Section 9 - 10.

 \begin{theorem}
We have the stable local trace formula:
   \[S^{G}_{\disc}(f)   = S^{G}(f) \]
 where as before
 \[S^{G}(f)=\sum_{M\in\mathcal{L}}|W^{M}_{0}||W^{G}_{0}|^{-1}(-1)^{\dimm(A_{M}/ A_{G})}\int_{\Delta_{G-\reg,\elll}(M,V,\zeta)}n(\delta)^{-1} S^G_{M}(\delta,f)d\delta,\]
 and
 \[S^{G}_{\disc}(f)=\int_{\Phi_{s-\disc}(G,\zeta)}|\mathcal{S}_{\phi}|^{-1}\sigma(\overline{S}_{\phi}^{\circ})f_{1}^{G}(\phi)\overline{f_{2}^{G}(\phi)}d\phi.\]

 \begin{proof}
 By induction on $\dimm(G_{\der})$. Put $\mathcal{E}_{\elll}^{\circ}(G) = \mathcal{E}_{\elll}(G) \backslash \{G\}$. Applying equation (7.11) and equation (8.2), we have
 \[ S^{G}(f)= \widehat{S}^G(f^{G})= I^G(f)-\sum_{G^{\prime}\in\mathcal{E}_{\elll}^{\circ}(G)}\iota(G,G^{\prime})\widehat{S}^{G^{\prime}}(f^{G^{\prime}})\]
 and
 \[S_{\disc}^{G}(f)= \widehat{S}^G_{\disc}(f^G)= I^G_{\disc}(f)-\sum_{G^{\prime}\in\mathcal{E}_{\elll}^{\circ}(G)}\iota(G,G^{\prime})\widehat{S}_{\disc}^{G^{\prime}}(f^{G^{\prime}}).\]

Now for $G^{\prime}\in\mathcal{E}_{\elll}^{\circ}(G)$, one has $\dimm(G^{\prime}_{\der})<\dimm(G_{\der})$, so by the induction hypothesis, we have
  \[\widehat{S}^{G^{\prime}}_{\disc}(f^{G^{\prime}})=\widehat{S}^{G^{\prime}}(f^{G^{\prime}}).\]
  Combining with the equality $I^G_{\disc}(f)=I^G(f)$, we have obtained the theorem.
 \end{proof}
 \end{theorem}

\bigskip

\begin{remark}
\end{remark}

\noindent (1) The method and results of this paper could be extended to the case where $G$ is a connected reductive group over a $p$-adic field, as long as the local Langlands correspondence is known for $G$, and such that the $R$-groups $R_{\phi}$ and the component groups $\mathcal{S}_{\phi}$ attached to the Langlands parameters of $G$ are all abelian. For instance the case of classical groups by the works \cite{A9,M} (with a slight complication in the even orthogonal case). We leave it to the reader to formulate the corresponding results. 

\bigskip

\noindent (2) 
For the geometric side of the stable local trace formula, the distributions $S^G_{M}(\delta,f)$ are defined inductively in terms of the invariant distributions $I^G_{M}(\gamma,f)$, which in turn is defined inductively by weighted orbital integrals, {\it c.f.} \cite{A4}. It is thus a highly non-trivial matter to obtain explicit formulas for the distributions $S^G_{M}(\delta,f)$, when $M \neq G$. In this regard, we refer the reader to \cite{P}, where the stable local trace formula is used to obtain explicit formulas for $S^G_{M}(\delta,f)$, in the case where $f= f_1 \times \bar{f}_2$, where $f_2$ is the stable pseudo-coefficient of a square integrable parameter of $G$.

\end{document}